\documentclass[12pt,reqno]{amsart}

\usepackage{amssymb}
\usepackage{amsmath}
\usepackage{latexsym}
\usepackage{amsthm}
\usepackage{epsfig}
\usepackage{color}
\usepackage{amsfonts,amssymb,mathrsfs,amscd}

\theoremstyle{plain}
\newtheorem{theorem}{Theorem}[section]
\newtheorem{proposition}{Proposition}[section]
\newtheorem{corollary}{Corollary}[section]
\newtheorem{lemma}{Lemma}[section]
\newtheorem{remark}{\bf Remark}[section]
\theoremstyle{definition}

\renewcommand{\Re}{\mathop\mathrm{Re}}
\newcommand{\D}{\mathrm{D}}

\title
[Interpolation inequalities for discrete
 operators ]
 {Sharp interpolation inequalities for discrete
 operators and applications}
\author[Alexei Ilyin, Ari Laptev, Sergey Zelik]
{Alexei Ilyin, Ari Laptev, Sergey Zelik}


\address
{\noindent\newline  Keldysh Institute of Applied Mathematics and
Institute for Information Transmission Problems;\newline
Imperial College London and Institute Mittag--Leffler;
\newline University of Surrey, Department of Mathematics and
Keldysh Institute of Applied Mathematics}
\email{ ilyin@keldysh.ru; a.laptev@imperial.ac.uk;
\newline s.zelik@surrey.ac.uk}

\begin{document}

\maketitle

\medskip

\bigskip
\begin{quote}
{\normalfont\fontsize{8}{10}
\selectfont{\bfseries  Abstract.}
We consider interpolation inequalities for
imbeddings of the $l^2$-sequence spaces over
$d$-dimensional lattices into  the $l^\infty_0$ spaces
written as interpolation inequality between the
$l^2$-norm of a sequence and its difference.
A general method is developed for finding sharp constants,
extremal elements and correction terms in this type of inequalities.
Applications to Carlson's inequalities and spectral theory
of discrete operators are given.
%
%
%
%
}
\end{quote}

\setcounter{equation}{0}
\section{Introduction}\label{S:Intro}

In this paper we study imbeddings of the sequence
space $l^2(\mathbb{Z}^d)$ into $l^\infty_0(\mathbb{Z}^d)$
written in terms of a interpolation inequality involving
the $l^2$-norms both of the sequence $u\in l^2(\mathbb{Z}^d)$,   and
the sequence of differences $\nabla u$, where for $u\in l^2(\mathbb{Z})$
and $n\in\mathbb{Z}$
$$
\D u(n)=u(n+1)-u(n),
$$
and for $ u\in l^2(\mathbb{Z}^d)$ and $n\in\mathbb{Z}^d$
$$
\nabla  u(n)=\{\D_1u(n),\dots,\D_du(n)\},
\quad \|\mathrm{D}u\|^2=\|\nabla u\|^2=\sum_{i=1}^d\|\D_iu\|^2.
$$

Before we describe the content of the paper in greater detail we give a
simple but important example \cite{Arman}, namely, let us prove
the one-dimensional inequality
\begin{equation}\label{1D}
\sup_nu(n)^2\le \|u\|\|\D u\|.
\end{equation}
The proof repeats that in the continuous case.
For an arbitrary $\gamma\in\mathbb{Z}$ we have
\begin{multline*}
2u^2(\gamma)=\biggl(\sum_{n=-\infty}^{\gamma-1}-\sum_{n=\gamma}^{\infty}\biggr)\D u^2(n)=\\=
\biggl(\sum_{n=-\infty}^{\gamma-1}-\sum_{n=\gamma}^{\infty}\biggr)
\bigl(u(n+1)\D u(n)+u(n)\D u(n)\bigr)\le\\\le
\sum_{n=-\infty}^\infty\left(|u(n+1)\D u(n)|+|u(n)\D u(n)|\right)\le
2\|u\|\|\D u\|.
\end{multline*}

Below we consider separately interpolation  inequalities
of the form
\begin{equation}\label{Dtheta}
   \sup_{\mathbf{n\in\mathbb{Z}}^d}
   u(\mathbf{n})^2\le \mathrm{K}_d(\theta)\|u\|^{2\theta}
   \|\nabla u\|^{2(1-\theta)}
,\quad 0\le\theta\le1
\end{equation}
in dimension $d=1,2$ and $d\ge3$. By notational definition
 $\mathrm{K}_d(\theta)$ is the sharp constant in this inequality.
This inequality clearly holds for $\theta=1$ (with
$\mathrm{K}_d(1)=1$), and if it holds for a $\theta=\theta_*\in[0,1)$,
then it holds for $\theta\in[\theta_*,1]$, when the
`weight' of the stronger norm $\|u\|$ is getting larger (see~\eqref{Poin_ineq}).

For $d=1$ we show that \eqref{Dtheta} holds for
$1/2\le\theta\le 1$ and find explicitly
the corresponding sharp constant:
\begin{equation}\label{K(theta)intro}
\mathrm{K}_1(\theta)=\frac12\left(\frac2\theta\right)^\theta
(2\theta-1)^{\theta-1/2}.
\end{equation}
In the limiting case $\theta=1/2$ we have $\mathrm{K}_1(1/2)=1$, and we supplement
 inequality \eqref{1D} (which is, in fact, sharp)
with a refined inequality
\begin{equation}\label{Lap1/2int}
u(0)^2\le\frac12\sqrt{4-\frac{\|\D u\|^2}{\|u\|^2}}\,
\|u\|\|\D u\|,
\end{equation}
which for any $d\in(0,4)$ has a unique extremal
sequence $u^*$ with $\|\D u^*\|^2/\|u^*\|^2=d$.

In the 2D case \eqref{Dtheta} holds for $0<\theta\le1$
and the sharp constant is given by
\begin{equation}\label{K(theta)2intro}
\mathrm{K}_2(\theta)=\frac2\pi\frac{1}{\theta^\theta(1-\theta)^{1-\theta}}\cdot
\max_{\lambda>0} \frac
{\lambda^{\theta}K\left(\frac{4}{{4+\lambda}}\right)}
{4+\lambda},
\end{equation}
where $K$ is the complete elliptic integral of the first
kind, see~\eqref{K}.
The constant $\mathrm{K}_2(\theta)$ logarithmically
tends to $\infty$ as $\theta\to 0^+$, and for $\theta=0$
we have the following limiting logarithmic inequality
of Brezis--Galluet type:
\begin{equation}\label{log-intro}
\aligned
u(0,0)^2\le \frac1{4\pi}\frac{\|\nabla u\|^2}{\| u\|^2}
\left(1-\frac{\|\nabla u\|^2}{8\| u\|^2}\right)
\left(\ln\frac{16}
{\frac{\|\nabla u\|^2}{\| u\|^2}
\left(8-\frac{\|\nabla u\|^2}{\| u\|^2}\right)}+\right.\\
+\left.\ln\left(1+\ln\frac{16}
{\frac{\|\nabla u\|^2}{\| u\|^2}
\left(8-\frac{\|\nabla u\|^2}{\| u\|^2}\right)}\right)+2\pi\right),
\endaligned
\end{equation}
where the constants in front of logarithms  and $2\pi$ are
sharp. The inequality saturates for $u=\delta$,
otherwise the inequality is strict.

Finally, in dimension three and higher the  inequality
holds for the limiting exponent $\theta=0$:
\begin{equation}\label{3Dintro}
u(0)^2\le \mathrm{K}_d\|\nabla u\|^2,
\end{equation}
where the sharp constant is given by
\begin{equation}\label{K3intro}
\mathrm{K}_d=\frac{1}{4(2\pi)^d}
\int_0^{2\pi}\!\dots\!\int_0^{2\pi}
\frac{dx_1\dots dx_d}{\sin^2\frac{x_1}2+\dots+
\sin^2\frac{x_d}2}\,.
\end{equation}
In the three dimensional case the constant
$\mathrm{K}_3$ can be evaluated in closed form since
it is expressed in terms of  the so-called third Watson's triple integral:
\begin{equation}\label{K3formula}
\mathrm{K}_3=\frac{1}{2}W_S=0.2527\dots,
\end{equation}
where (see \cite{lattice_sums} and the references therein)
\begin{equation}\label{W_S}
\aligned
W_S:=\frac1{\pi^3}\int_0^\pi\int_0^\pi\int_0^\pi
\frac{dxdydz}{3-\cos x-\cos y-\cos z}=\\=
\frac{\sqrt{6}}{12(2\pi)^3}
\Gamma(\tfrac1{24})\Gamma(\tfrac5{24})\Gamma(\tfrac7{24})\Gamma(\tfrac{11}{24}).
\endaligned
\end{equation}

It is natural to compare interpolation inequalities for
differences and inequalities for derivatives in the continuous case.
While in the continuous case the $L_\infty$-norm is the strongest
(at least locally), in the discrete case the $l^1$-norm is
the strongest. Obviously, $\|u\|_{l^\infty}\le\|u\|_{l^p}$ for $p\ge1$,
and therefore $\|u\|_{l^p}\le\|u\|_{l^q}$ for $q\le p$:
$$
\|u\|_{l^p}^p\le\|u\|_{l^\infty}^{p-q}\|u\|_{l^q}^q\le
\|u\|_{l^p}^{p-q}\|u\|_{l^q}^q.
$$
Also, unlike the continuous case, the difference operator is
\textit {bounded\/}:
\begin{equation}\label{Poin_ineq}
\|\mathrm{D}u\|_{l^2(\mathbb{Z}^d)}^2\le4d\|u\|_{l^2(\mathbb{Z}^d)}^2.
\end{equation}

Roughly speaking, the situation
(at least in the one-dimensional case) is as follows.
The discrete inequality \eqref{Dtheta} for $d=1$ holds
for $\theta\in[1/2,1]$, while the corresponding
continuous inequality
$$
 \|f\|_\infty^2\le \mathrm{C}_1(\theta)\|f\|^{2\theta}
   \|f'\|^{2(1-\theta)}
,\quad f\in H^1(Q)
$$
holds only for $\theta=1/2$ in case when $Q=\mathbb{R}$,
and for $\theta\in[0,1/2]$ for periodic function with zero mean,
 $Q=\mathbb{T}^1$.
Hence, it makes sense to compare the constants at a unique
common point $\theta_*=1/2$ where both constants are equal to $1$.
For $n$-order   derivatives and differences, $n>1$,
the constants in the discrete inequalities are strictly greater
than those in the continuous case, the corresponding $\theta_*=1-1/(2n)$.

For example, the second-order inequality on the line $\mathbb{R}$
and the corresponding discrete inequality are as follows
$$
\aligned
\|f\|_{L_\infty(\mathbb{R})}^2&\le \frac{\sqrt{2}}{\sqrt[4]{27}}\|f\|^{3/2}
   \|f''\|^{1/2}
,\quad f\in H^2(\mathbb{R}),\\
\|u\|_{l^\infty(\mathbb{Z})}^2&\le \frac{\sqrt{2}}2\|u\|^{3/2}
   \|\Delta u\|^{1/2}
,\quad u\in l^2(\mathbb{Z}).
\endaligned
$$
Both constants are sharp, the second one is strictly greater than the first.
Up to a constant factor (and shift of the origin) the family of extremal
functions in the first inequality is produced by scaling
 $x\to\lambda x$, $\lambda>0$ of the extremal $f_*( x)$, where
$$
\int_{-\infty}^\infty\frac{e^{-ixy}\,dx}{x^4+1}=
\frac{\pi\sqrt{2}}2f_*\biggl(\frac{y}{\sqrt{2}}\biggr),\qquad
f_*(x)=e^{-|x|}(\cos x+\sin|x|),
$$
In the discrete inequality the unique extremal sequence is
$\{u_*(n)\}_{n=-\infty}^\infty$,
$$
u_*(n)=\int_{0}^\pi\frac{\cos nx\,dx}{\lambda_*+16\sin^4\frac x2},
\quad \text{where }\quad \lambda_*=\frac{16}3;
$$
see~\eqref{expl2ord} for the explicit formula for $u_*(n)$.

In two dimensions in the continuous case the imbedding
$H^1\subset L_\infty$ holds only with a logarithmic
correction term involving higher Sobolev norms (and  $\theta=0$), which is the well-known
Brezis--Gallouet inequality. On the contrary, in the $2D$
discrete case inequality \eqref{Dtheta} holds for
$\theta\in(0,1]$ and also requires a logarithmic correction
for $\theta=0$, see~\eqref{log-intro}.

In higher
dimensional case $d\ge3$ the imbedding $H^1\subset L_\infty$
fails at all, while inequality \eqref{Dtheta} holds for all
$\theta\in[0,1]$.

Next, we consider applications of discrete interpolation inequalities.
Using the discrete Fourier transform and Parseval's identities we
show that each discrete interpolation inequality is equivalent to an
integral Carslon-type inequality. For example, in the  1D case, setting
for  a function $g\in L_2(0,2\pi)$
$$
I_1:=\int_0^{2\pi}g(x)dx,\ I_2^2:=\int_0^{2\pi}g(x)^2dx,
\ \hat{I}_2^2:=\int_0^{2\pi}4\sin^2\tfrac x2g(x)^2dx,
$$
we obtain that inequality~\eqref{1D} is equivalent to  the sharp inequality
$$
I_1^2\le2\pi I_2\hat{I}_2,
$$
with no extremal functions, while the refined inequality~\eqref{Lap1/2int}
is equivalent to the inequality
$$
I_1^2\le\pi\sqrt{4-\frac{\hat{I}_2^2}{{I}_2^2}}\, I_2\hat{I}_2,
$$
saturating for each $\lambda\in(-\infty,-4)\cup(0,\infty)$ at
$$
g_\lambda(x)=\frac1{\lambda+4\sin^2\frac x2}\,.
$$

Developing further this approach we prove a Sobolev $l^q$-type discrete
inequality for a non-limiting exponent
\begin{equation}\label{Sobolev}
\|u\|_{l^q(\mathbb{Z}^d)}\le \mathrm{C}(q,d)\|\nabla u\|
\quad\text{for}\quad q>{2d}/(d-2).
\end{equation}
Our explicit estimate for the  constant $\mathrm{C}(q,d)$ is non-sharp, moreover, it
blows up as $q\to {2d}/(d-2)$ however, it is sharp in the
limit $q\to\infty$.

Finally, we apply the results on discrete inequalities to
the estimates of negative eigenvalues of discrete Schr\"odinder
operators
\begin{equation}\label{Schr_int}
-\Delta -V
\end{equation}
acting in $l^2(\mathbb{Z}^d)$. Here $-\Delta:=\mathrm{D}^*\mathrm{D}$ and
$V(n)\ge0$. Each discrete interpolation inequality
for the imbedding into $l^\infty(\mathbb{Z}^d)$ produces by the method
of~\cite{E-F} a collective inequality for families of orthonormal
sequences, which, in turn, is equivalent to a Lieb-Thirring
estimate for the negative trace.
For example, we deduce from \eqref{3Dintro} the estimate
$$
\sum_{\lambda_j<0}|\lambda_j|\le \frac{\mathrm{K}_d}{4}\sum_{\alpha\in\mathbb{Z}^d}
V^{2}(\alpha),
$$
which holds for $d\ge3$.

We finally point out that in the continuous case the classical
Lieb--Thirring inequality for the negative trace of
operator~\eqref{Schr_int} in $L_2(\mathbb{R}^d)$ is
as follows
(see \cite{L-T}, \cite{Lap-Weid}, \cite{D-L-L})
$$
\sum_{\lambda_j<0}|\lambda_j|\le \mathrm{L}_{1,d}
\int_{\mathbb{R}^d}V^{1+d/2}(x)dx.
$$

\setcounter{equation}{0}
\section{1D case}\label{S:1D}

Since $u_n\to0$ as $|n|\to\infty$, without loss of generality we can assume that $\sup_n u(n)^2=u(0)^2$.

We consider a more general problem of finding
sharp constants, existence of extremals and possibly correction terms in the inequalities of the type
\begin{equation}\label{1Dtheta}
   u(0)^2\le \mathrm{K}_1(\theta)\|u\|^{2\theta}
   \|\D u\|^{2(1-\theta)}
,\quad 0\le\theta\le1,
\end{equation}
including, to begin with, the problem of
finding those $\theta$ for which \eqref{1Dtheta}
holds at all. Here
$$\|u\|^2=
\sum_{k=-\infty}^{\infty}
u(k)^2,\quad
\|\D u\|^2=\sum_{k=-\infty}^{\infty}
\D u(k)^2.
$$

Since $|a-b|\ge||a|-|b||$,
we have
\begin{equation}\label{D|u|}
\|\mathrm{D}|u|\|\le\|\mathrm{D}u\|,
\quad\text{where}\quad
|u|:=\{|u(n)|\}_{n=-\infty}^\infty,
\end{equation}
and  we could have further reduced
our treatment to the case when $u(n)\ge0$. However,
we shall be dealing below with a more general problem~\eqref{V(d)} which has  both sing-definite and non-sign-definite extremals.  We have the following `reverse' Poincare inequality:
\begin{equation}\label{Poincare}
\|\mathrm{D}u\|^2\le\left\{
\begin{array}{ll} 2\|u\|^2, & u \hbox{ is sign definite;} \\
4\|u\|^2, & \hbox{otherwise.}
\end{array}
\right.
\end{equation}

The adjoint to $\D$ is the operator:
$$
\D^*u(n)=-(u(n)-u(n-1)),
$$
and
$$
\D^*\D u(n)=\D\D^*u(n)=-\bigl(u(n+1)-2u(n)+u(n-1)\bigr).
$$

To find the sharp constant $\mathrm{K}_1(\theta)$
in~\eqref{1Dtheta} we consider a more general problem:
find $\mathbb{V}(d)$, where
$\mathbb{V}(d)$ is the solution of the following
maximization problem:
\begin{equation}\label{V(d)}
\mathbb{V}(d):=\sup\bigl\{u(0)^2:\
u\in l^2(\mathbb{Z}), \ \|u\|^2=1,\
\|\D u\|^2=d\bigr\},
\end{equation}
where $0<d<4$.

Its solution is found in terms of the Green's
function of the corresponding second-order self-adjoint positive operator,
see \cite{IZ}, \cite{Zelik}.
The spectrum of the operator $-\Delta=\D^*\D$
is the closed interval $[0,4]$,
and we set
\begin{equation}\label{A(lam)}
\mathbb{A}(\lambda)=
\left\{
  \begin{array}{ll}
\phantom{-} \D^*\D+\lambda, & \hbox{for $\lambda>0$;} \\
-\D^*\D-\lambda, & \hbox{for $\lambda<-4$.}
  \end{array}
\right.
\end{equation}
Then $\mathbb{A}(\lambda)$ is positive definite
$$
(\mathbb{A}(\lambda)u,u)=
\left\{
  \begin{array}{ll}
\phantom{-} \|\D u\|^2+\lambda\|u\|^2>\lambda\|u\|^2, & \hbox{for $\lambda>0$;} \\
-\|\D u\|^2-\lambda\|u\|^2>(-\lambda-4)\|u\|^2, & \hbox{for $\lambda<-4$.}
  \end{array}
\right.
$$
Let $\delta$ be the delta-sequence: $\delta(0)=1$,
$\delta(n)=0$ for $n\ne0$, and let $G_\lambda=\{G_\lambda(n)\}_{n=-\infty}^\infty\in l^2(\mathbb{Z})$ be the Green's function of operator \eqref{A(lam)}, that is, the solution of the equation:
\begin{equation}\label{Green}
\mathbb{A}(\lambda)G_\lambda=\delta.
\end{equation}

Then we have by the Cauchy--Schwartz inequality
\begin{equation}\label{additive}
\aligned
u(0)^2=(\delta,u)^2=(\mathbb{A}(\lambda)G_\lambda,u)^2=
\\=
(\mathbb{A}(\lambda)^{1/2}G_\lambda,\mathbb{A}(\lambda)^{1/2}u)^2\le
(\mathbb{A}(\lambda)G_\lambda,G_\lambda)
(\mathbb{A}(\lambda)u,u)=\\=G_\lambda(0)
(\mathbb{A}(\lambda)u,u).
\endaligned
\end{equation}
Furthermore, this inequality is sharp and  turns into equality
if and only if $u=\mathrm{const}\cdot G_\lambda$.

We find in Lemma~\ref{L:explicit}  explicit formulas for
$\mathbb{V}(d)$ and $G_\lambda(n)$. Nevertheless,
we now independently prove the following two
symmetry properties of $\mathbb{V}(d)$ and $G_\lambda(n)$,
especially since their counterparts will be useful
in the two-dimensional case below.
\begin{proposition}\label{P:symm}
For $d\in(0,4)$
\begin{equation}\label{Vsymm}
\mathbb{V}(d)=\mathbb{V}(4-d).
\end{equation}
For $\lambda>0$ and $n\in\mathbb{Z}$
\begin{equation}\label{Gsymm}
0<G_\lambda(n)=(-1)^{|n|}G_{-4-\lambda}(n).
\end{equation}
\end{proposition}
\begin{proof} For $u\in l^2(\mathbb{Z})$ we define the
orthogonal operator $T$
$$
Tu=u^\star:=\{(-1)^{|n|}u(n)\}_{n=-\infty}^\infty.
$$
Then clearly $\|u\|^2=\|u^\star\|^2$ and, in addition,
\begin{equation}\label{D*-D}
\|\mathrm{D}u^\star\|^2=4\|u\|^2-\|\mathrm{D}u\|^2.
\end{equation}
Therefore if for a fixed $d$ and $u=u_d$ we have
$$
\mathbb{V}(d)=u(0)^2, \quad \|u\|^2=1\quad
\text{and}\quad \|\mathrm{D}u\|^2=d,
$$
then for $u^*=Tu$ it holds
$$
u^\star(0)^2=u(0)^2=\mathbb{V}(d), \quad \|u^\star\|^2=1\quad
\text{and}\quad \|\mathrm{D}u^\star\|^2=4-d,
$$
which gives that $\mathbb{V}(4-d)\ge\mathbb{V}(d)$. However, the strict
inequality here is impossible, since otherwise by repeating this
procedure we would have found that $\mathbb{V}(d)>\mathbb{V}(d)$.
This proves~\eqref{Vsymm}.

Turning to \eqref{Gsymm} we note that $T^{-1}=T^*=T$ and
we see from~\eqref{D*-D} that
$$
(\mathrm{D}^*\mathrm{D}u,u)=(T(-\mathrm{D}^*\mathrm{D}+4)Tu,u),
$$
and, consequently,
$$
\mathrm{D}^*\mathrm{D}=T(-\mathrm{D}^*\mathrm{D}+4)T.
$$
Therefore, if for $\lambda>0$, $G_\lambda$ solves
$$
\mathbb{A}(\lambda)G_\lambda=(\mathrm{D}^*\mathrm{D}+\lambda)G_\lambda=
\delta,
$$
then
$$
T(-\mathrm{D}^*\mathrm{D}+4+\lambda) TG_\lambda=\delta.
$$
Since $T^{-1}\delta=\delta$, using  definition~\eqref{A(lam)}
we obtain
$$
(-\mathrm{D}^*\mathrm{D}+4+\lambda)TG_\lambda=
\mathbb{A}(-4-\lambda)TG_\lambda=\delta,
$$
which gives
$$
TG_\lambda=G_{-4-\lambda},
$$
and proves the equality in~\eqref{Gsymm}.

It remains to  show that for $\lambda>0$  $G_\lambda(n)>0$ for all $n$.
Since $\mathbb{A}(\lambda)$ is positive definite,
it follows that $G_\lambda(0)=(\mathbb{A}(\lambda)G_\lambda,G_\lambda)>0$.
 We use the maximum principle and
suppose that for some $n\ne1$, $G_\lambda(n)<0$. Since
$G_\lambda(n)\to0$ as $n\to\infty$ and $G_\lambda(0)>0$, it follows that
$G_\lambda$ attains
a global strictly negative minimum at some point $n>1$
(the case $n<-1$ is similar).
Then the sum of the first three terms in~\eqref{recur}
is non-positive and  the fourth term is strictly negative,
which contradicts  $\delta(n)=0$. This proves that
$G_\lambda(n)\ge0$ for all $n$. Finally, to prove strict
positivity, we suppose that $G_\lambda(n)=0$ for some  $n>1$.
Then we see from~\eqref{recur} that
$G_\lambda(n-1)+G_\lambda(n+1)=0$, and what has already
 been proved
gives $G_\lambda(n-1)=G_\lambda(n+1)=0$. Repeating this
we reach $n=1$ giving that $G_\lambda(0)=0$, which is
a contradiction.
\end{proof}

To denote the three norms of $G_\lambda$ we  set
\begin{equation}\label{fgh1}
f(\lambda):=G_\lambda(0),\quad
g(\lambda):=\|G_\lambda\|^2,\quad
h(\lambda):=\|\D G_\lambda\|^2.
\end{equation}

\begin{lemma}\label{L:relations}
The functions $f$, $g$ and $h$ satisfy
\begin{equation}\label{fgh}
g(\lambda)=-\mathrm{sign}(\lambda)f'(\lambda),\ \
 h(\lambda)=\mathrm{sign}(\lambda)(f(\lambda)+\lambda f'(\lambda)).
\end{equation}
\end{lemma}
\begin{proof} Let $\lambda>0$.
Then $\mathbb{A}(\lambda)=\D^*\D+\lambda$. Taking the scalar product of \eqref{Green} with $G_\lambda$
we have
\begin{equation}\label{4.green}
f(\lambda)=G_\lambda(0)=\|\D G_\lambda\|^2+
\lambda\|G_\lambda\|^2=h(\lambda)+\lambda g(\lambda).
\end{equation}
Differentiating this formula with
respect to $\lambda$  we obtain
\begin{multline}\label{4.diff}
f'(\lambda)=2(\mathbb{A}(\lambda)G'_\lambda,G_\lambda)+
g(\lambda)=\\=
-2(G_\lambda,G_\lambda)+g(\lambda)=-g(\lambda),
\end{multline}
where we used that $G_\lambda+\mathbb{A}(\lambda)G'_\lambda=0$,
which, in turn, follows from~\eqref{Green}. The case $\lambda<-4$
is treated similarly taking into account that
now $\mathbb{A}(\lambda)=-\D^*\D-\lambda$.
\end{proof}
\begin{corollary}\label{C:d}
The function $d(\lambda)$ defined as follows
\begin{equation}\label{d(lambda)}
d(\lambda):=\frac{\|\D G_\lambda\|^2}{\|G_\lambda\|^2}
=\frac{h(\lambda)}{g(\lambda)}
\end{equation}
satisfies the functional equation
\begin{equation}\label{d(4-lambda)}
d(-4-\lambda)=4-d(\lambda).
\end{equation}
\end{corollary}
\begin{proof} It follows from \eqref{Gsymm} and \eqref{fgh1} that
$$
f(-4-\lambda)=f(\lambda).
$$
Hence, $f'(\lambda)=-f'(-4-\lambda)$ and we obtain from~\eqref{fgh}
$$
\aligned
d(-4-\lambda)=\frac{h(-4-\lambda)}{g(-4-\lambda)}=
\frac{f(-4-\lambda)+(-4-\lambda)f'(-4-\lambda)}{-f'(-4-\lambda)}=\\=
\frac{f(\lambda)+(4+\lambda)f'(\lambda)}{f'(\lambda)}=
\frac{f(\lambda)+\lambda f'(\lambda)}{f'(\lambda)}+4=
4-d(\lambda).
\endaligned
$$
\end{proof}

Next, we find explicit formulas for $f$, $g$ and $h$.
\begin{lemma}\label{L:explicit}
The Green's function $G_\lambda$ belongs
to $l^2(\mathbb{Z})$, and
both for $\lambda\in(-\infty,-4)$ and
$\lambda\in(0,\infty)$
\begin{equation}\label{expl}
f(\lambda)=\frac1{\sqrt{\lambda(\lambda+4)}},\
g(\lambda)=\frac{\lambda+2}{(\lambda+4)\sqrt{\lambda^3(\lambda+4)}},\
h(\lambda)=\frac2{\sqrt{\lambda(\lambda+4)^3}}\,.
\end{equation}
Furthermore,  the elements $G_\lambda(n)$  can be found explicitly: for $\lambda>0$
\begin{equation}\label{explG(n)}
G_\lambda(n)=\frac1\pi\ \int_{0}^\pi\frac{\cos nx\,dx}{\lambda+4\sin^2\frac x2}=
\frac1{\sqrt{\lambda(\lambda+4)}}
\left(\frac{\lambda+2-\sqrt{\lambda(\lambda+4)}}2\right)^{|n|}\,,
\end{equation}
for $\lambda<-4$
\begin{equation}\label{explG(n)-}
G_\lambda(n)=-\frac1\pi\ \int_{0}^\pi\frac{\cos nx\,dx}{\lambda+4\sin^2\frac x2}=
\frac1{\sqrt{\lambda(\lambda+4)}}
\left(\frac{\lambda+2+\sqrt{\lambda(\lambda+4)}}2\right)^{|n|}\,.
\end{equation}

\end{lemma}
\begin{proof} In view of \eqref{fgh}, for the proof
of~\eqref{expl} it suffices to find only $f(\lambda)=G_\lambda(0)$. We consider two cases:
$\lambda>0$ and $\lambda<-4$.
For $\lambda>0$ the sequence
$G_\lambda$ solves \eqref{Green}, which takes the
form $(\D^*\D+\lambda)G_\lambda=\delta$, or
component-wise
\begin{equation}\label{recur}
-G_\lambda(n+1)+2G_\lambda(n)-G_\lambda(n-1)+
\lambda G_\lambda(n)=\delta(n).
\end{equation}
We multiply each equation by $e^{inx}$ and sum the results from $n=-\infty$ to $\infty$. Setting
$$
\widehat{g}_\lambda(x):=\sum_{n=-\infty}^\infty
G_\lambda(n)e^{inx},
$$
we obtain
$$
-\sum_{n=-\infty}^\infty
e^{inx}G_\lambda(n+1)+(2+\lambda)\sum_{n=-\infty}^\infty
e^{inx}G_\lambda(n)-\sum_{n=-\infty}^\infty
e^{inx}G_\lambda(n-1)=1
$$
or
$$
\aligned
1=\widehat{g}_\lambda(x)(\lambda-e^{-ix}+2-e^{ix}))=\\=
\widehat{g}_\lambda(x)\bigl(\lambda-(e^{ix/2}-e^{-ix/2})^2\bigr)=
\widehat{g}_\lambda(x)\left(\lambda+4\sin^2\frac x2\right),
\endaligned
$$
which gives
$ \widehat{g}_\lambda(x)=1/(\lambda + 4\sin^2\frac x2)$.

In the case when $\lambda<-4$ equation \eqref{Green} becomes $(\D^*\D+\lambda)G_\lambda=-\delta$ and we merely have to change the sign of $\widehat{g}_\lambda(x)$ and we obtain:
\begin{equation}\label{ghat}
\widehat{g}_\lambda(x)=\left\{
\begin{array}{ll}
\phantom{-}\frac1{\lambda + 4\sin^2\frac x2}, & \lambda>0; \\
-\frac1{\lambda + 4\sin^2\frac x2}                           , & \lambda<-4,
                         \end{array}
                       \right.
\end{equation}
and
\begin{equation}\label{G(n)}
G_\lambda(n)=\frac1{2\pi}\int_{-\pi}^\pi\widehat{g}_\lambda(x)\cos nx\,dx.
\end{equation}
Since $\widehat{g}_\lambda\in L_2(0,2\pi)$, it follows that
$G_\lambda\in l^2(\mathbb{Z})$.

Using the integral
\begin{equation}\label{integral}
\int_{-\pi}^\pi\frac{dx}{b+\sin^2\frac x2}=
\left\{
  \begin{array}{ll}
\phantom{-}\frac{2\pi}{\sqrt{b(b+1)}}, & b>0; \\
-\frac{2\pi}{\sqrt{b(b+1)}}, & b<-1.
  \end{array}
\right.
\end{equation}
we finally obtain both for $\lambda>0$,  and $\lambda<-4$
\begin{equation}\label{G(0)}
f(\lambda)=G_\lambda(0)=\frac1{2\pi}\int_{-\pi}^\pi
\widehat{g}_\lambda(x)dx=
\frac1{\sqrt{\lambda(\lambda+4)}}\,.
\end{equation}

Finally, to obtain the explicit formula \eqref{explG(n)}
(which will not be used below) we observe that  the
equation~\eqref{recur} for positive (and negative) $n$
is a homogeneous linear recurrence relation with constant coefficients.
The characteristic equation is
$$
q^2-(2+\lambda)q+1=0
$$
with  roots
$$
\aligned
q_1(\lambda)&=\frac{\lambda+2-\sqrt{\lambda(\lambda+4)}}2,
\quad 0<q_1(\lambda)<1\ \text{ for}\ \lambda>0,\\
q_2(\lambda)&=\frac{\lambda+2+\sqrt{\lambda(\lambda+4)}}2,
\quad -1<q_2(\lambda)<0\ \text{ for}\ \lambda<-4.
\endaligned
$$
For $\lambda>0$  the general
$l^2$-solution of~\eqref{recur}
is $G_\lambda(n)=c_1(\lambda)q_1(\lambda)^n$ for $n>0$ and
$G_\lambda(n)=c_2(\lambda)q_1(\lambda)^{|n|}$ for $n<0$. Since we already
know that $G_\lambda(n)=G_\lambda(-n)$, it follows that
$c_1(\lambda)=c_2(\lambda)=:a(\lambda)$. Substituting
$G_\lambda(n)=a(\lambda)q_1(\lambda)^{|n|}$ into \eqref{recur} with $n=0$
we obtain $-2a(\lambda)q_1(\lambda)+(2+\lambda)a(\lambda)=1$,
which gives
$$
a(\lambda)=\frac1{2+\lambda-2q_1(\lambda)}=\frac1{\sqrt{\lambda(\lambda+4)}}\,
$$
and proves~\eqref{explG(n)}.
The proof of~\eqref{explG(n)-} in the case $\lambda<-4$ is totally similar, we only have to use the second root
$q_2(\lambda)$ with $|q_2(\lambda)|<1$.

We finally point out that the equality~\eqref{Gsymm} can now be
also verified by a direct calculation: $q_2(-4-\lambda)=-q_1(\lambda)$.
\end{proof}

We can now give the solution to the problem  \eqref{V(d)}.
\begin{theorem}\label{T:V(d)}
For any $0<d<4$ the solution of the maximization problem~\eqref{V(d)} is given by
\begin{equation}\label{Sol_V(d)}
\mathbb{V}(d)=\frac12\sqrt{d(4-d)}.
\end{equation}
The supremum in~\eqref{V(d)} is the maximum that is attained at a unique sequence
$u^*_{\lambda(d)}=\|G_{\lambda(d)}\|^{-1}G_{\lambda(d)}$, where
\begin{equation}\label{lambda(d)}
\lambda(d)=\frac{2d}{{2-d}}
\end{equation}
for $d\ne2$;  for $d=2$, $u^*=\delta$.
\end{theorem}
\begin{proof}
It follows from \eqref{additive} that
for any $u\in l^2(\mathbb{Z})$
$$
u(0)^2\le G_\lambda(0)(\mathbb{A}(\lambda)u,u),
$$
and, furthermore,  for
$$
u_\lambda^*:=\frac1{\|G_\lambda\|}\cdot G_\lambda
$$
with $\|u_\lambda^*\|^2=1$ the above inequality
turns into equality.

Next, using \eqref{expl} we find the formula for the function $d(\lambda)$
defined in~\eqref{d(lambda)}
$$
d(\lambda)=\frac{\|\D G_\lambda\|^2}{\|G_\lambda\|^2}
=\frac{h(\lambda)}{g(\lambda)}=\frac{2\lambda}{2+\lambda}\,, \qquad d:\left\{
    \begin{array}{ll}
    (0,\infty)\to(0,2),   \\
     (-\infty,-4)\to(2,4).
    \end{array}
  \right.
$$
The inverse function $\lambda(d)$ is given by~\eqref{lambda(d)} and with this $\lambda(d)$
we have
$$
\frac{\|\D u_{\lambda(d)}^*\|^2}{\|u_{\lambda(d)}^*\|^2}
=\frac{h(\lambda(d))}{g(\lambda(d))}=
\frac{2\lambda(d)}{2+\lambda(d)}=d\,.
$$
Therefore $u_{\lambda(d)}^*$ is the extremal sequence in~\eqref{V(d)} and its solution
is
$$
\mathbb{V}(d)=u_{\lambda(d)}^*(0)^2=
\frac{f(\lambda(d))^2}{g(\lambda(d))}=
\frac12\sqrt{d(4-d)}.
$$
\end{proof}

\begin{remark}\label{R:Vandd}
{\rm
It is worth pointing out that in accordance with
Proposition~\ref{P:symm} and Corollary~\ref{C:d} we directly see
here that
$\mathbb{V}(d)=\mathbb{V}(4-d)$ and $d(-4-\lambda)=4-d(\lambda)$.
For the inverse function $\lambda(d)$ we have the functional equation
$\lambda(4-d)=-4-\lambda(d)$.
}
\end{remark}

\begin{corollary}\label{C:1/2}
For any $u\in l^2(\mathbb{Z})$
inequality~\eqref{1D} holds, the constant~$1$ is sharp
and no extremals exist. The following refined inequality holds:
\begin{equation}\label{Lap1/2}
u(0)^2\le\frac12\sqrt{4-\frac{\|\D u\|^2}{\|u\|^2}}\,
\|u\|\|\D u\|.
\end{equation}
For any $0<d<4$ the inequality saturates for
$u^*_{\lambda(d)}=G_{\lambda(d)}$, where
$\lambda(d)=\frac{2d}{{2-d}}$
(see~\eqref{lambda(d)})
with
$\|\D u^*_{\lambda(d)}\|^2/
\|u^*_{\lambda(d)}\|^2=d$. For $d=2$,
$u^*=\delta$.
\end{corollary}
\begin{proof}
Inequality~\eqref{Lap1/2} follows from
\eqref{Sol_V(d)} by homogeneity.

Since $\mathbb{V}(d)<\sqrt{d}$, we obtain inequality~\eqref{1D},
and since $\mathbb{V}(d)/\sqrt{d}\to1$ as $d\to0$ the constant
$1$ is sharp.
In view of the refined inequality \eqref{Lap1/2} there can be no extremals in
the original inequality~\eqref{1D}.
\end{proof}
We now consider \eqref{1Dtheta} for $\theta\ne 1/2$.
\begin{theorem}\label{T:theta}
Inequality  \eqref{1Dtheta} holds only for $1/2\le\theta\le1$. The sharp constant $\mathrm{K}_1(\theta)$ is
\begin{equation}\label{K(theta)}
\mathrm{K}_1(\theta)=\frac12\left(\frac2\theta\right)^\theta
(2\theta-1)^{\theta-1/2}.
\end{equation}
For each $1/2<\theta\le1$ there exists a unique extremal sequence.
\end{theorem}
\begin{proof} The proof is similar to the proof of
 Theorem~2.5 in~\cite{IZ} where the classical Sobolev spaces were
considered. For convenience we include some details.

We first observe that inequality~\eqref{1Dtheta} cannot hold for
$\theta<1/2$, since otherwise we would have found that
$\mathbb{V}(d)\le c d^\eta$, $\eta=1-\theta>1/2$, a contradiction
with~\eqref{Sol_V(d)}: $\mathbb{V}(d)\sim d^{1/2}$ as
$d\to0$.

The case $\theta=1/2$ was treated above and we assume in what follows  that
$\theta>1/2$.
We set
\begin{equation}\label{lambda}
\lambda:=\frac{\theta}{1-\theta}
\frac{\|\D u\|^2}{\|u\|^2}.
\end{equation}
Then, using \eqref{additive}, we have
\begin{equation}\label{sup}
\aligned
u(0)^2\le G_\lambda(0)\|u\|^2\left(\frac{\|\D
u\|^2}{\|u\|^2}+ \lambda\right)
= \frac1\theta
G_\lambda(0)\lambda^\theta\lambda^{1-\theta}\|u\|^2=
\\=\frac1\theta\lambda^\theta G_\lambda(0)
\left(\frac{\theta}{1-\theta}\frac{\|\D u\|^2}{\|u\|^2}\right)^{1-\theta}
\|u\|^2=\\= \frac{1}{\theta^\theta(1-\theta)^{1-\theta}}\cdot\lambda^{\theta}
G_\lambda(0)\|u\|^{2\theta}\|\D u\|^{2(1-\theta)}\le\\
\le \frac{1}{\theta^\theta(1-\theta)^{1-\theta}}\cdot
\sup_{\lambda>0}\bigg\{\lambda^{\theta} G_\lambda(0)\bigg\}\,
\|u\|^{2\theta}\|\D u\|^{2(1-\theta)}=\\= \mathrm{K}_1(\theta)\|u\|^{2\theta}\|\D
u\|^{2(1-\theta)}.
\endaligned
\end{equation}

We have taken into account in the last equality that
$$
G_\lambda(0)=f(\lambda)=
\frac1{\sqrt{\lambda(\lambda+4)}}.
$$
Hence, the supremum in the above formula is
a (unique) maximum on $\lambda\in\mathbb{R}^+$ of the function
$$
\lambda^\theta f(\lambda)=\frac{\lambda^{\theta-1/2}}{\sqrt{\lambda+4}}
$$
attained at $\lambda_*=(4\theta-2)/(1-\theta)$,
which gives \eqref{K(theta)}.
To see that the constant $\mathrm{K}_1(\theta)$ is sharp we use that
 $\frac{d}{d\lambda}(\lambda^\theta
f(\lambda))|_{\lambda=\lambda_*}=0$, and
$\lambda_*f'(\lambda_*)+\theta f(\lambda_*)=0$.
In view of \eqref{fgh} this gives
$$
d_*:=d(\lambda_*)=\frac{\|\D G_{\lambda_*}\|^2}{\|G_{\lambda_*}\|^2}
=\frac{h(\lambda_*)}{g(\lambda_*)}=
-\frac{f(\lambda_*)+\lambda_* f'(\lambda_*)}{f'(\lambda_*)}
=\frac{1-\theta}\theta\lambda_*.
$$
Hence \eqref{lambda} is satisfied for $u_*=G_{\lambda_*}$
 the two inequalities in \eqref{sup} become equalities,
and $u_*$ is the unique extremal.
\end{proof}

\begin{figure}[htb]
\centerline{\psfig{file=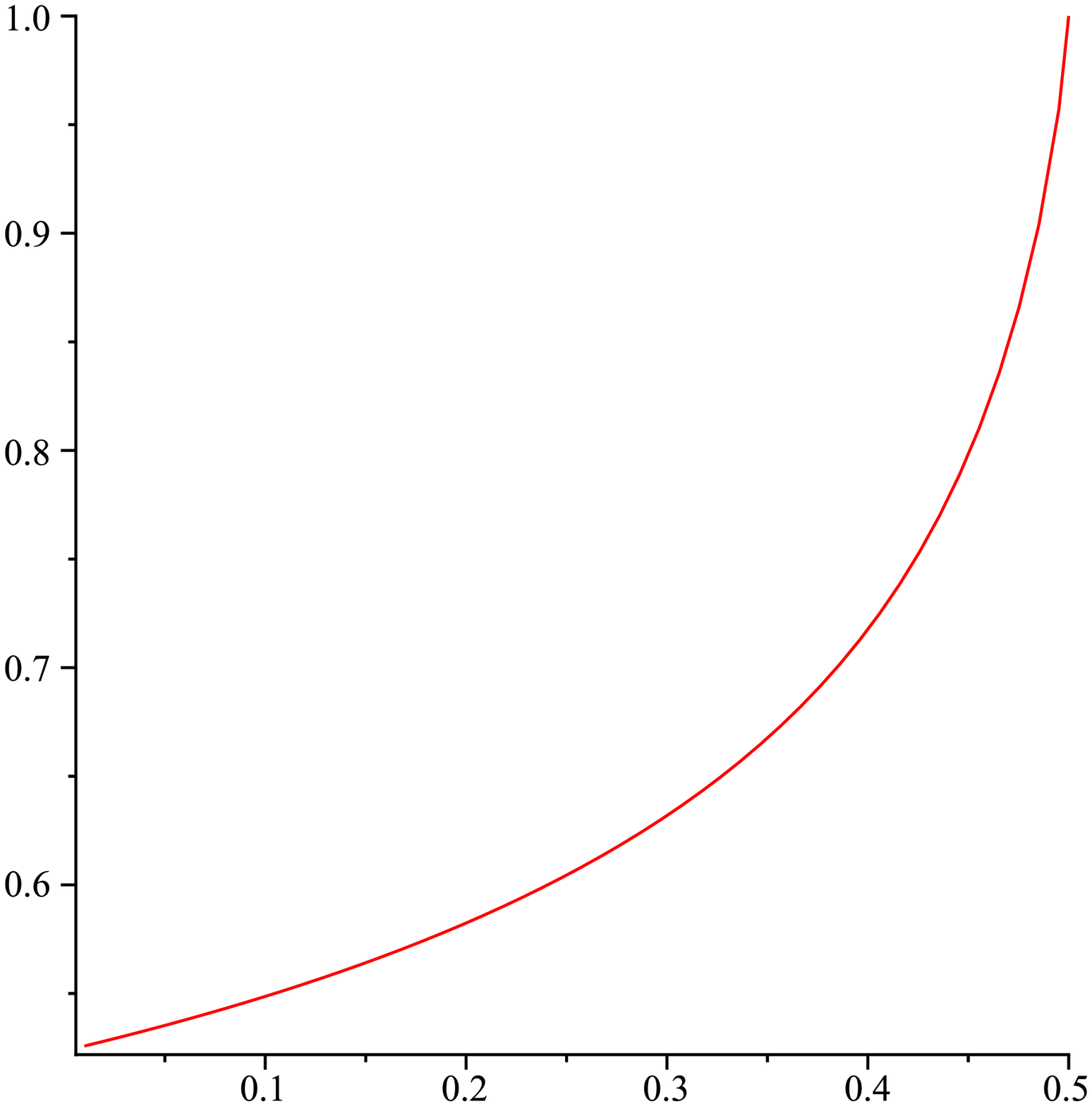,width=6cm,height=5cm,angle=0}
\psfig{file=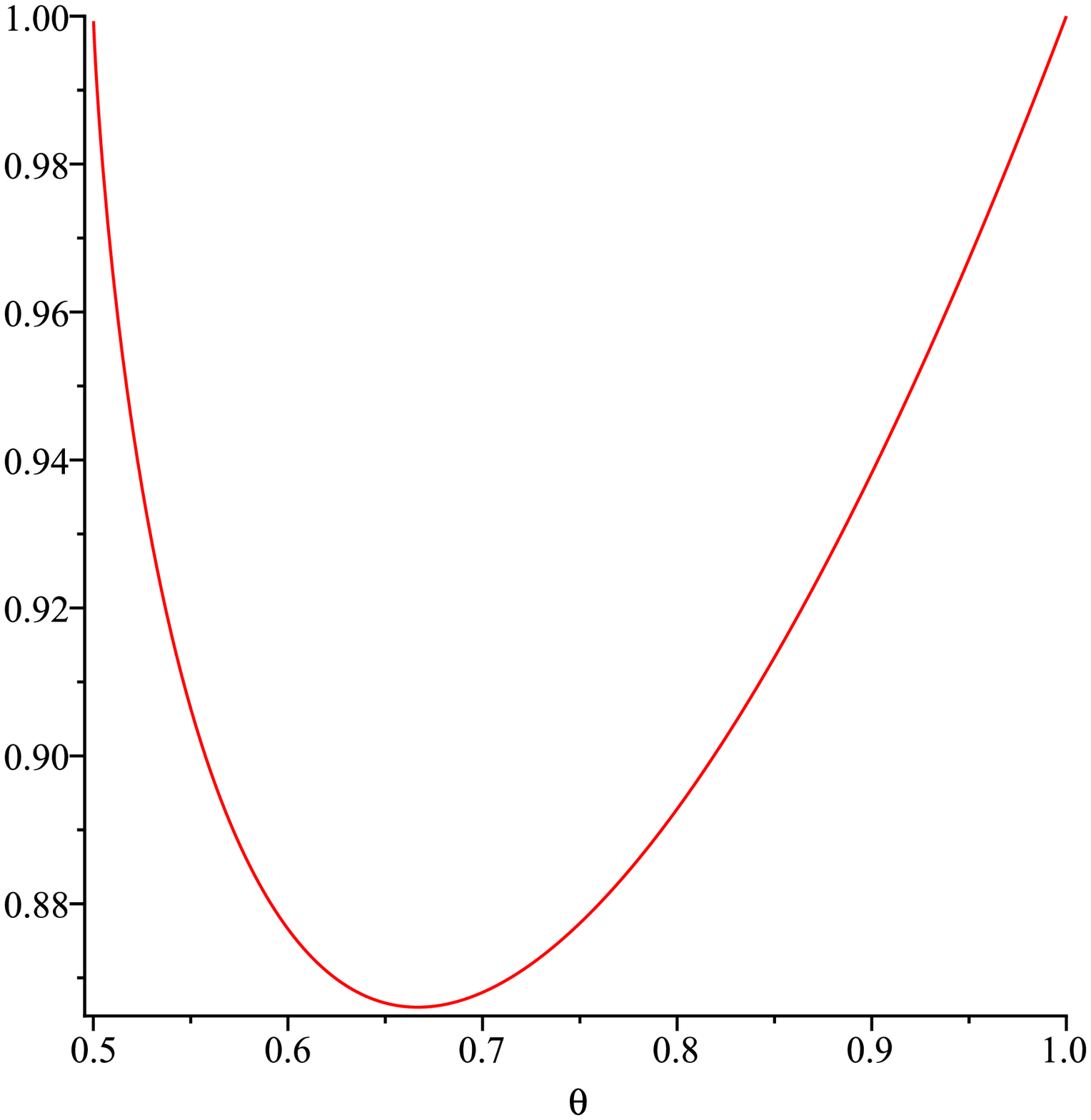,width=6cm,height=5cm,angle=0}
}
\caption{Graphs of sharp constants in one-dimensional  first-order inequalities
on complementary intervals: periodic functions (left)
\eqref{C11theta}, discrete case (right) \eqref{K(theta)}.}
\label{fig:K(theta)}
\end{figure}

The graph of the function $\mathrm{K}_1(\theta)$ is shown
in Fig.\ref{fig:K(theta)} on the right.
Here $\mathrm{K}_1(1/2)=1$ corresponds to~\eqref{1D}, and
$\mathrm{K}_1(1)=1$ corresponds to the trivial inequality
$u(0)^2\le\|u\|^2$ with extremal $u=\delta$.

\begin{remark}\label{R:expr_for_V}
{\rm
In this theorem we do not use the formula~\eqref{Sol_V(d)} for
$\mathbb{V}(d)$. However, if we do, then
finding~$\mathrm{K}_1(\theta)$ for $\theta\in[1/2,1]$
becomes very easy. In fact, by the definition of~$\mathbb{V}(d)$
and homogeneity, $\mathrm{K}_1(\theta)$ is the smallest constant
 for which
$\mathbb{V}(d)\le \mathrm{K}_1(\theta)d^{1-\theta}$
for all  $ d\in[0,4]$. Therefore
$$
\mathrm{K}_1(\theta)=\max_{d\in[0,4]}\mathbb{V}(d)/d^{1-\theta}=
\max_{d\in[0,4]}\frac12 d^{\theta-1/2}(4-d)^{1/2}=
\text{r.h.s.}\eqref{K(theta)}.
$$
The corresponding $d_*=(4\theta-2)/\theta\le2$ and $\lambda(d_*)>0$,
see~\eqref{lambda(d)}. This also explains why the region of
negative $\lambda$ does not play a role in Theorem~\ref{T:theta}.
}
\end{remark}

\setcounter{equation}{0}
\section{2D case}\label{S:2D}

In this  section we consider the two-dimensional inequalities
\begin{equation}\label{2Dtheta}
   u(0,0)^2\le \mathrm{K}_2(\theta)\|u\|^{2\theta}
   \|\nabla u\|^{2(1-\theta)}
,\quad 0\le\theta\le1,
\end{equation}
and address the same problems as in the previous section.

We set
$$
\Delta=-\D_1^*\D_1-\D_2^*\D_2.
$$
Then $(-\Delta u,u)=\|\nabla u\|^2\le8\|u\|^2$
for $u\in l_2(\mathbb{Z}^2)$. As in the 1D case we shall be dealing with the following extremal problem:
\begin{equation}\label{V(d)2}
\mathbb{V}(d):=\sup\bigl\{u(0,0)^2:\
u\in l^2(\mathbb{Z}^2), \ \|u\|^2=1,\
\|\D u\|^2=d\bigr\},
\end{equation}
where $0<d<8$.

The resolvent set of $-\Delta+\lambda$
is $(-\infty,-8)\cup(0,\infty)$ and
as before we consider the positive self-adjoint operator operator
$$
\mathbb{A}(\lambda)=
\left\{
  \begin{array}{ll}
-\Delta+\lambda, & \hbox{for $\lambda>0$;} \\
\phantom{-}\Delta-\lambda, & \hbox{for $\lambda<-8$.}
  \end{array}
\right.
$$
Our main goal is to find the Green's function of it:
\begin{equation}\label{Green2}
\mathbb{A}(\lambda)G_\lambda=\delta,
\end{equation}
more precisely, $G_\lambda(0,0)$.

\begin{proposition}\label{P:symm2}
For $d\in(0,8)$
\begin{equation}\label{Vsymm2}
\mathbb{V}(d)=\mathbb{V}(8-d).
\end{equation}
For $\lambda>0$ and $(n,m)\in\mathbb{Z}^2$
\begin{equation}\label{Gsymm2}
0<G_\lambda(n,m)=(-1)^{|n+m|}G_{-8-\lambda}(n,m).
\end{equation}
Finally, the function $d(\lambda)=\frac{\|\nabla G_\lambda\|^2}{\|G_\lambda\|^2}$ satisfies
\begin{equation}\label{dsymm2}
d(-8-\lambda)=8-d(\lambda).
\end{equation}
\end{proposition}
\begin{proof} The proof is completely analogous
to that of Proposition~\ref{P:symm} and
Corollary~\ref{C:d}, where the functions $f(\lambda)$,
$g(\lambda)$ and $h(\lambda)$ have the same meaning as in
\eqref{fgh1} and satisfy~\eqref{fgh}. The operator $T$ is
as follows
$$
Tu(n,m)=(-1)^{|n+m|}u(n,m).
$$
\end{proof}
\begin{lemma}\label{L:Green2(0,0)}
For $\lambda\in(-\infty,-8)\cup(0,\infty)$ the Green's function $G_\lambda\in l^2(\mathbb{Z}^2)$
 and
\begin{equation}\label{Green2(0,0)}
G_\lambda(0,0)=\frac2\pi\frac
{K\left(\frac{4}{{4+\lambda}}\right)}
{|4+\lambda|},
\end{equation}
where $K(k)$ is the complete elliptic integral of the first kind:
\begin{equation}\label{K}
K(k)=\int_0^1\frac{dt}{\sqrt{(1-t^2)(1-k^2t^2)}}\,.
\end{equation}
\end{lemma}
\begin{proof}
Setting
$$
\widehat{g}_\lambda(x,y):=\sum_{k,l=-\infty}^\infty
G_\lambda(k,l)e^{ikx+ily},
$$
and acting as in Lemma~\ref{L:explicit} we find that
$$
\widehat{g}_\lambda(x,y)=\frac{\mathrm{sign}(\lambda)}{\lambda +
4(\sin^2\frac x2+\sin^2\frac y2)}\,
$$
and
\begin{equation}\label{G(n,m)}
G_\lambda(n,m)=\frac{\mathrm{sign}(\lambda)}{4\pi^2}
\int_0^{2\pi}\int_0^{2\pi}\frac{\cos(nx+my)dxdy}
{\lambda+4(\sin^2\frac x2+\sin^2\frac y2)}\,.
\end{equation}

Therefore for $\lambda>0$, using \eqref{integral}
\begin{equation}\label{gradshtein}
\aligned
G_\lambda(0,0)=
\frac14\frac1{4\pi^2}\int_{0}^{2\pi} dx\int_{0}^{2\pi}
\frac{dy}{(\frac\lambda 4+
\sin^2\frac x2)+\sin^2\frac y2}=\\=
\frac1{16\pi^2}\int_{0}^{2\pi}
\frac{2\pi dx}{\sqrt{(\frac\lambda 4+\sin^2\frac x2)
(\frac\lambda 4+1+\sin^2\frac x2)}}=\\=
\frac1{4\pi}\int_{0}^\pi
\frac{ dx}{\sqrt{(\frac\lambda 4+\sin^2\frac x2)
(\frac\lambda 4+1+\sin^2\frac x2)}}=\\=
\frac1{4\pi}\int_{0}^1\frac{ dt}{\sqrt{t(1-t)(\frac\lambda 4+t)
(\frac\lambda 4+1+t)}}=\\=
\frac1{4\pi}\cdot \frac{2K\left(\frac{1}{\frac\lambda 4+1}\right)}{\frac\lambda 4+1},
\endaligned
\end{equation}
where the last integral was calculated by
transforming general elliptic integrals to the standard form (see formula $3.147.7$ in \cite{Gr-Ry}).

Since $K(k)$ is even, we see from \eqref{Gsymm2} that formula
\eqref{Green2(0,0)} works both for $\lambda>0$ and $\lambda<-8$.
\end{proof}
\begin{remark}\label{R:change}
{\rm
The equality in \eqref{Gsymm2} also follows from
\eqref{G(n,m)} by changing the variables
$(x,y)\to(x'+\pi,y'+\pi)$ and using the fact that the
integrand is even.
}
\end{remark}

\begin{theorem}\label{T:theta2}
The inequality
$$
u(0,0)^2\le \mathrm{K}_2(\theta)\|u\|^{2\theta}
   \|\nabla u\|^{2(1-\theta)}
$$
holds for $\theta\in (0,1]$.
For $\theta\in (0,1)$ the sharp constant $\mathrm{K}_2(\theta)$ is
\begin{equation}\label{K(theta)2}
\mathrm{K}_2(\theta)=\frac{1}{\theta^\theta(1-\theta)^{1-\theta}}\cdot
\max_{\lambda>0}(\lambda^{\theta} G_\lambda(0,0)),
\end{equation}
and for each $0<\theta<1$ there exists a unique extremal sequence
\begin{equation}\label{extr2}
u_{\lambda_*}=G_{\lambda_*},
\quad\text{where}\quad
\lambda_*=\mathrm{argmax}(\lambda^{\theta} G_\lambda(0,0)).
\end{equation}
Finally, $\mathrm{K}_2(1)=1$ with $u_*=\delta$,
and
\begin{equation}\label{to0}
\mathrm{K}_2(\theta)=\frac1{4\pi e\theta}+
o\left(\frac1\theta\right)\quad
\text{as}\quad\theta\to 0^+.
\end{equation}
The graph of the function $\mathrm{K}_2(\theta)$
 is shown in~Fig~\ref{fig:K2(theta)}.
\begin{figure}[htb]
\centerline{\psfig{file=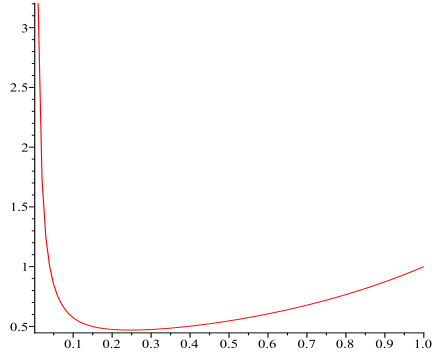,width=6cm,height=5cm,angle=0}}
\caption{Graph of $\mathrm{K}_2(\theta)$ on $\theta\in (0,1]$.}
\label{fig:K2(theta)}
\end{figure}
\end{theorem}
\begin{proof}
Similarly to
Theorem~\ref{T:theta}, we have
\begin{equation}\label{gen_form}
\mathrm{K}_2(\theta)=\frac{1}{\theta^\theta(1-\theta)^{1-\theta}}\cdot
\sup_{\lambda>0}(\lambda^{\theta} G_\lambda(0,0)),
\end{equation}
where, of course, $G_\lambda(0,0)$ is given by~\eqref{Green2(0,0)}. We have the following
asymptotic expansions
\begin{equation}\label{asymptotic}
\aligned
&
G_\lambda(0,0)=\frac1{4\pi}(2\log(4\sqrt{2})+
\log(\tfrac1\lambda))+O(\lambda\log(\tfrac1\lambda)),
\quad\text{as}\quad \lambda\to0,\\
&
G_\lambda(0,0)=\frac1\lambda-\frac4{\lambda^2}+
O(\tfrac1{\lambda^3}),
\quad\text{as}\quad \lambda\to\infty.
\endaligned
\end{equation}
Hence, for $0<\theta<1$ we see that  $\lambda^\theta G_\lambda(0,0)=0$ both at $\lambda=0$ and $\lambda=\infty$, and the supremum in~\eqref{gen_form} is the maximum, which proves
\eqref{K(theta)2} and \eqref{extr2}.

We also see from the first formula that for
small positive $\theta$ the leading term in
the second factor in \eqref{gen_form}
is
$$
\max_{0<\lambda<1}\frac1{4\pi}\lambda^\theta\log(\tfrac1\lambda)
=\frac1{4\pi e\theta},
$$
while the first factor tends to $1$. This proves~\eqref{to0}.
For example,
$$
\mathrm{K}_2(0.01)=3.205\dots,
\ \text{while}\
\frac1{0.99^{0.99}0.01^{0.01}}\frac1{4\pi e \cdot 0.01}=
3.096\dots\,.
$$

\end{proof}

In the limiting case $\theta=0$ inequality \eqref{2Dtheta} holds
with a logarithmic correction term of Brezis--Galouet
type \cite{BG},\cite{Zelik}.

The solution of the extremal problem~\eqref{V(d)2} is given in
terms of the functions $f(\lambda)$, $g(\lambda)$ and
$h(\lambda)$:
\begin{equation}\label{fgh2d}
\aligned
&f(\lambda)=G_\lambda(0,0)=\frac2\pi\frac
{K\left(\frac{4}{{4+\lambda}}\right)}
{|4+\lambda|}\,,\\
&g(\lambda)=\|G_\lambda\|^2=
-\mathrm{sign}(\lambda)f'(\lambda)=\frac{2E(\frac4{4+\lambda})}{\pi\lambda(\lambda+8)}\,,\\
&h(\lambda)=\|\nabla G_\lambda\|^2=\mathrm{sign}(\lambda)\bigl(f(\lambda)+\lambda f'(\lambda)\bigr)
=\frac2\pi\frac {K(\frac4{4+\lambda})}{4+\lambda}-
\frac{2E(\frac4{4+\lambda})}{\pi(\lambda+8)}\,,
\endaligned
\end{equation}
where  $E(k)$ is the complete elliptic integral of the second kind:
$$
E(k)=\int_0^1 \frac{\sqrt{1-k^2 t^2} }{\sqrt{1-t^2}}dt,
$$
and where we used
$\frac{\mathrm{d}K(k)}{\mathrm{d}k}=\frac{E(k)}{k(1-k^2)}-\frac{K(k)}{k}$.

\begin{theorem}\label{T:V(d)2}
The solution $\mathbb{V}(d)$ of problem~\eqref{V(d)2}
is
\begin{equation}\label{V(d)2sol}
\mathbb{V}(d)=u^*_{\lambda(d)}(0)^2=\frac {f^2(\lambda(d))}{g(\lambda(d))}=\frac {2K(\frac 4{4+\lambda(d)})^2\lambda(d)(\lambda(d)+8)}
{\pi(4+\lambda(d))^2E(\frac4{4+\lambda(d)})},
\end{equation}
where $\lambda(d)$ is the inverse function of
the function $d(\lambda)$:
\begin{equation}\label{d(lambda)2}
d(\lambda)=\frac{h(\lambda)}{g(\lambda)}=\frac{\lambda(\lambda+8)}{(\lambda+4)}
\frac{K(\frac4{\lambda+4})}{E(\frac4{\lambda+4})}-\lambda,
\end{equation}
and where $u^*_{\lambda(d)}=G_\lambda/\|G_\lambda\|$.
Here $d(\lambda)$ is defined on $(-\infty,-8)\cup(0,\infty)$,
satisfies \eqref{dsymm2}
and monotonically increases from
$d(-\infty)=4$ to $d(-8)=8$ and then from
$d(0)=0$ to $d(\infty)=4$. The inverse function
$\lambda(d)$ is defined on $d\in[0,8]\setminus\{0\}$ and satisfies
$$
\lambda(8-d)=-8-\lambda(d).
$$
Their graphs are shown in Fig.~\ref{Fig:d(l)+l(d)}.
Finally, $\mathbb{V}(4)=1$ and $u^*=\delta$.
\begin{figure}[htb]
\centerline{\psfig{file=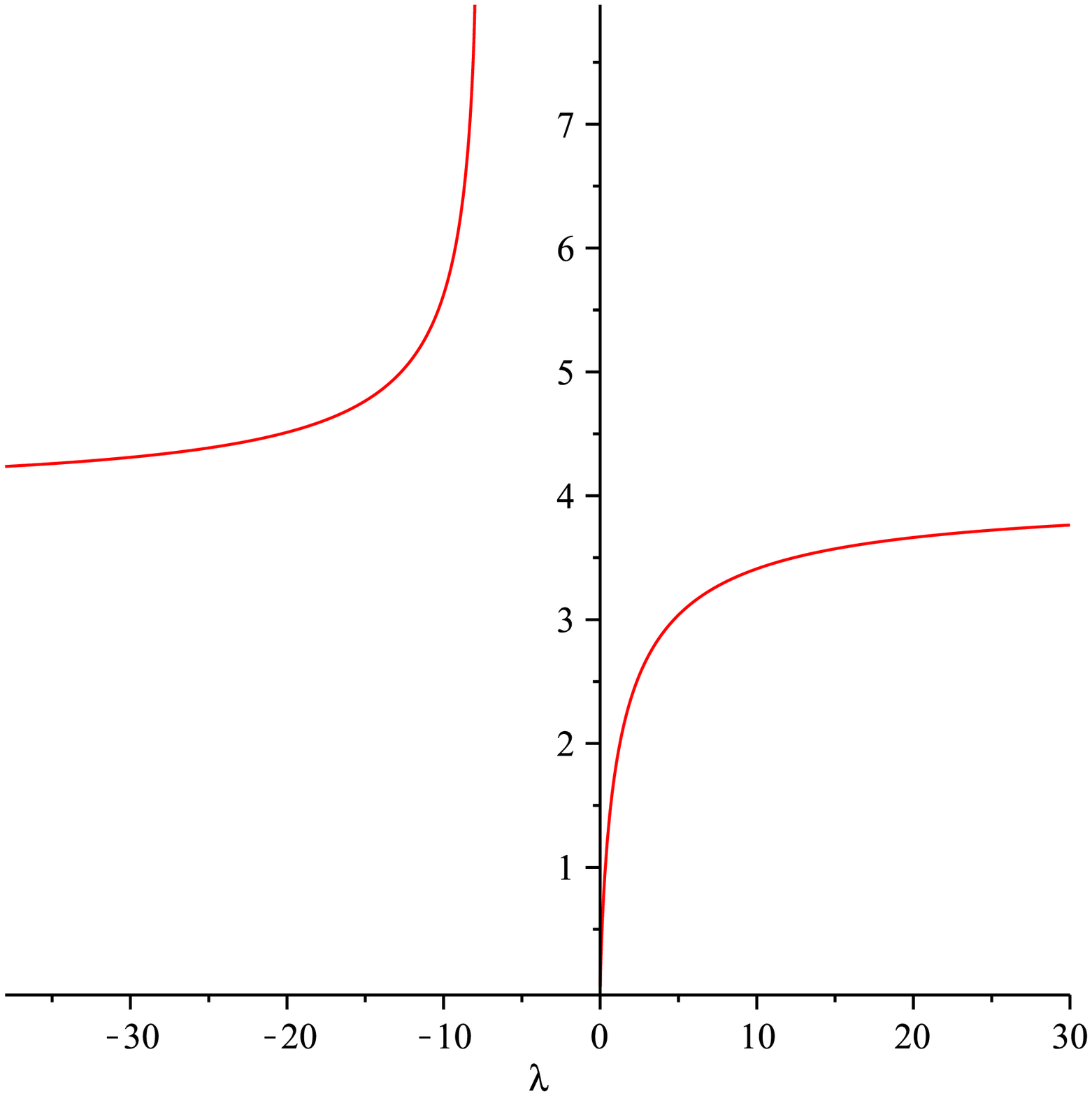,width=6.4cm,height=5cm,angle=0}
\psfig{file=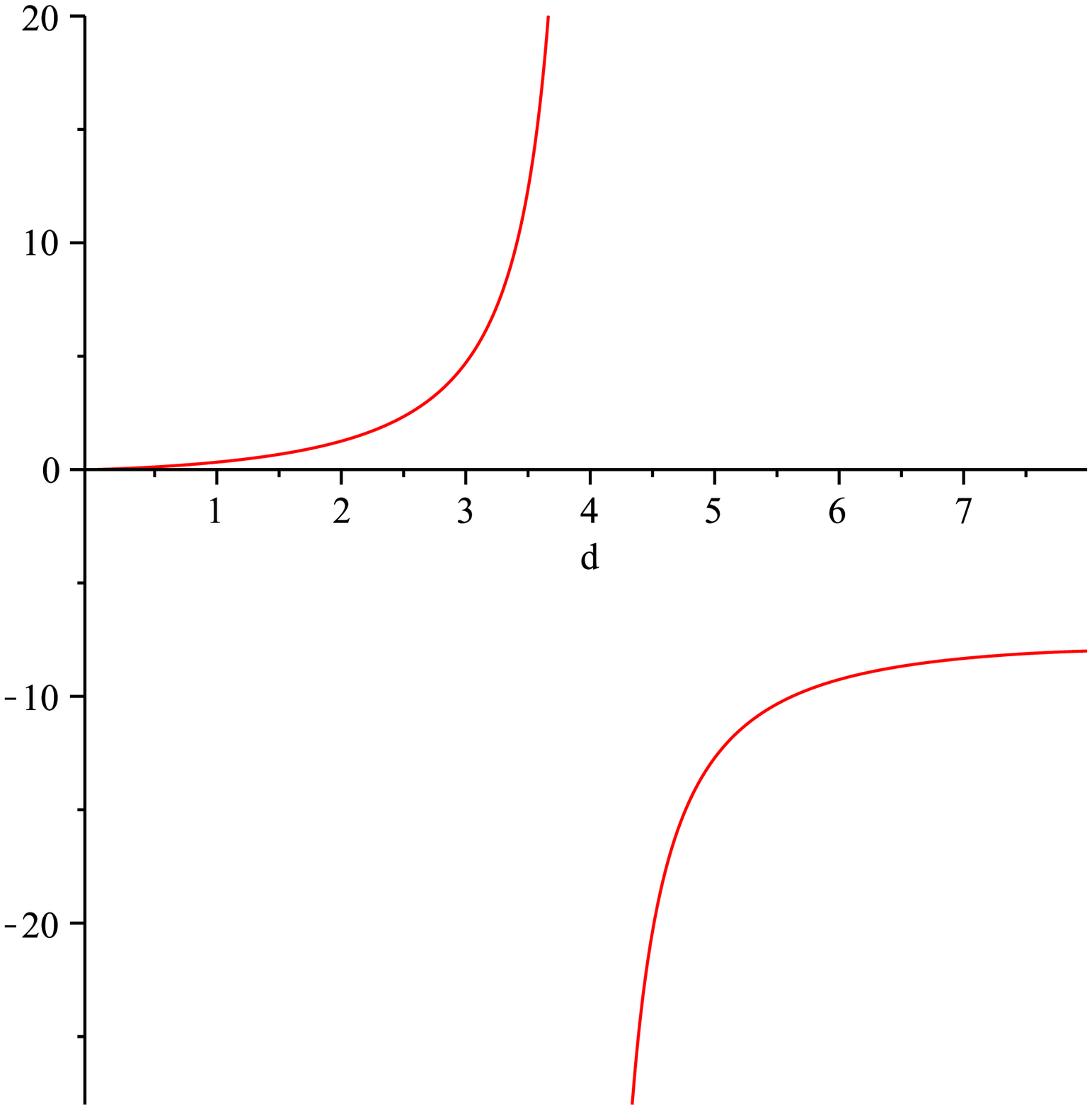,width=6.4cm,height=5cm,angle=0} }
\caption{Graphs of $d(\lambda)$ and $\lambda(d)$.} \label{Fig:d(l)+l(d)}
\end{figure}
\end{theorem}

\begin{proof}
We act as in Theorem~\ref{T:V(d)}, the essential difference
being that we now do not have a formula for the inverse
function $\lambda(d)$, by means  of which we construct the extremal element for
each $d$.
Although $d(\lambda)$ is given explicitly, the monotonicity
of it  required for the existence of the inverse function
is a rather general fact and can be verified
as in \cite[Theorem~2.1]{IZ}, where the continuous case was considered.
\end{proof}

We now find an explicit majorant $\mathbb{V}_0(d)$
for the implicitly defined solution $\mathbb{V}(d)$.
In view of the symmetry~\eqref{Vsymm2}
 it suffices  to study the case $d\to0$
only and then, by replacing $d\to d(8-d)/8$ we get
the symmetric expansions valid for both singularities.
We have the following expansions
\begin{equation}\label{d(lambda)0infty}
\aligned
&d(\lambda)=(5\ln2-\ln\lambda-1)\lambda+O_{\lambda\to0}((\lambda\ln\lambda)^2),\\
&d(\lambda)=4-\frac8{\lambda}+O_{\lambda\to\infty}(1/\lambda^2).
\endaligned
\end{equation}
Truncating the first expansion and solving
$d=(5\ln2-\ln\lambda-1)\lambda$, we have
$$
\lambda=-\frac{d}{W_{-1}(-\frac1{32}ed)}
$$
where $W_{-1}(z)$ is the $-1$th branch of the
Lambert function. Using the known asymptotic
expansions for the Lambert function, we get the following
expression for $\lambda(d)$
\begin{equation}\label{lamda(d)expan}
\lambda(d)=\frac d{5\ln2-1-\ln(d)+\ln(5\ln(2)-1-\ln d)
+O(\frac{\ln(-\ln d)}{\ln d})}.
\end{equation}
Using
\begin{equation}\label{lambda=0and infty}
\aligned
&\frac {f^2(\lambda)}{g(\lambda)}=
\frac{(\frac52\ln2-\frac12\ln\lambda)^2}\pi
\lambda+O_{\lambda\to0}(\lambda^2(\ln\lambda)^3),\\
&\frac {f^2(\lambda)}{g(\lambda)}=1-\frac4{\lambda^2}+
O_{\lambda\to\infty}\biggl(\frac1{\lambda^3}\biggr)
\endaligned
\end{equation}
and substituting \eqref{lamda(d)expan} into the first
expansion we  get
$$
\mathbb V(d)=\frac1{4\pi}d(-\ln d+\ln(1-\ln d)+
\gamma+o_{d\to0}(1))
$$
where $\gamma=5\ln(2)+1<2\pi$. This
justifies our choice of the approximation to
$\mathbb V(d)$:
$$
\mathbb V_0(d)=\frac1{4\pi} \frac {d(8-d)}8
\left(\ln\frac{16}{d(8-d)}+
\ln\biggl(1+\ln\frac{16}{d(8-d)}\biggr)+
2\pi\right)
$$
The constant $2\pi$ instead of $\gamma$
(and the numerator $16$) are  chosen so
that for  $d=4$ we have $\mathbb{V}(4)=\mathbb{V}_0(4)=1$.

The asymptotic expansion of $\mathbb V_0(d)$ at $d=0$
shows that $\mathbb V(d)<\mathbb V_0(d)$
for $0<d\le d_0$, where  $d_0$ is sufficiently small.

Using  the expansions at $\lambda=\infty$ in
\eqref{d(lambda)0infty} and \eqref{lambda=0and infty}
we find that
$$
\mathbb{V}(d)=1-\frac{(4-d)^2}{16}+O_{d\to 4}((4-d)^3).
$$
Since
$$
\mathbb{V}_0(d)=1-\left(1-\frac1\pi\right)\frac{(4-d)^2}{16}+O_{d\to 4}((4-d)^3),
$$
it follows that
$
\mathbb V(d)\le V_0(d)
$
for  $d\in[4-d_1,4]$ for a small $d_1>0$.
Corresponding to $[d_0,d_1]$ is the {\it finite}
interval $[\lambda_0,\lambda_1]$ on which
computer calculations show that the inequality
$\mathbb V(d)\le V_0(d)$ still holds. This gives that
$$
\mathbb V(d)\le V_0(d)
$$
for all $d\in[0,4]$ and hence, by symmetry, for
$d\in[0,8]$.

Thus, we have proved the
following inequality.
\begin{theorem}\label{T:loglog}
For $u\in l_2(\mathbb{Z}^2)$
\begin{multline}
u(0,0)^2\le \frac1{4\pi}\frac{\|\nabla u\|^2}{\| u\|^2}
\left(1-\frac{\|\nabla u\|^2}{8\| u\|^2}\right)
\left(\ln\frac{16}
{\frac{\|\nabla u\|^2}{\| u\|^2}
\left(8-\frac{\|\nabla u\|^2}{\| u\|^2}\right)}+\right.\\
+\left.\ln\left(1+\ln\frac{16}
{\frac{\|\nabla u\|^2}{\| u\|^2}
\left(8-\frac{\|\nabla u\|^2}{\| u\|^2}\right)}\right)+2\pi\right),
\end{multline}
where the constants in front of logarithms  and $2\pi$ are
sharp. The inequality saturates for $u=\delta$,
otherwise the inequality is strict.
\end{theorem}

\setcounter{equation}{0}
\section{3D case}\label{S:3D}

In the three-dimensional case the following result
holds which is somewhat similar to the classical Sobolev inequality for the limiting exponent.
\begin{theorem}\label{T:theta3}
Let $u\in l^2(\mathbb{Z}^3)$. Then for any $\theta\in[0,1]$
\begin{equation}\label{3D}
u(0,0,0)^2\le \mathrm{K}_3(\theta)\|u\|^{2\theta}
\|\nabla u\|^{2(1-\theta)},
\end{equation}
where $\mathrm{K}_3(\theta)<\infty$ for $\theta\in[0,1]$,
and its sharp value for $\theta\in(0,1)$ is given by
\begin{equation}\label{in(0,1)}
\mathrm{K}_3(\theta)=\frac1{2\pi^2}
\frac{1}{\theta^\theta(1-\theta)^{1-\theta}}
\max_{\lambda>0}\lambda^\theta
\int_0^\pi\frac{K\left(
\frac{1}{\frac\lambda 4+1+\sin^2\frac x2}\right)}
{\frac\lambda 4+1+\sin^2\frac x2}dx,
\end{equation}
and  there exists a unique extremal
element, which belongs to $ l^2(\mathbb{Z}^3)$.

In the limiting case $\theta=0$ inequality~\eqref{3D}
still holds:
\begin{equation}\label{3D0}
u(0,0,0)^2\le \mathrm{K}_3(0)
\|\nabla u\|^{2},
\end{equation}
where
\begin{equation}\label{K3}
\aligned
\mathrm{K}_3(0)=\frac1{(2\pi)^3}\int_{\mathrm{T}^3}
\frac{dx}{4(\sin^2\frac {x_1}2+\sin^2\frac {x_2}2+\sin^2\frac {x_3}2)}=\\=\frac{1}{2\pi^2}\cdot
\int_0^\pi\frac{K\left(\frac{1}{1+\sin^2\frac x2}\right)}{1+\sin^2\frac x2}dx=
\frac{4.9887\dots}{2\pi^2}=0.2527\dots.
\endaligned
\end{equation}
The constant is sharp and there exists a unique
extremal element, which does not lie in $l^2(\mathbb{Z}^3)$,
but rather in $l^\infty_0(\mathbb{Z}^3)$,
but whose gradient does belong to $l^2(\mathbb{Z}^3)$.
Furthermore, as we already mentioned in \S\ref{S:Intro},
we have the closed form formula for $\mathrm{K}_3(0)$ (see
\cite{{lattice_sums}})
$$
\mathrm{K}_3(0)=\frac{\sqrt{6}}{24(2\pi)^3}
\Gamma(\tfrac1{24})\Gamma(\tfrac5{24})\Gamma(\tfrac7{24})\Gamma(\tfrac{11}{24}).
$$
\end{theorem}
\begin{figure}[htb]
\centerline{\psfig{file=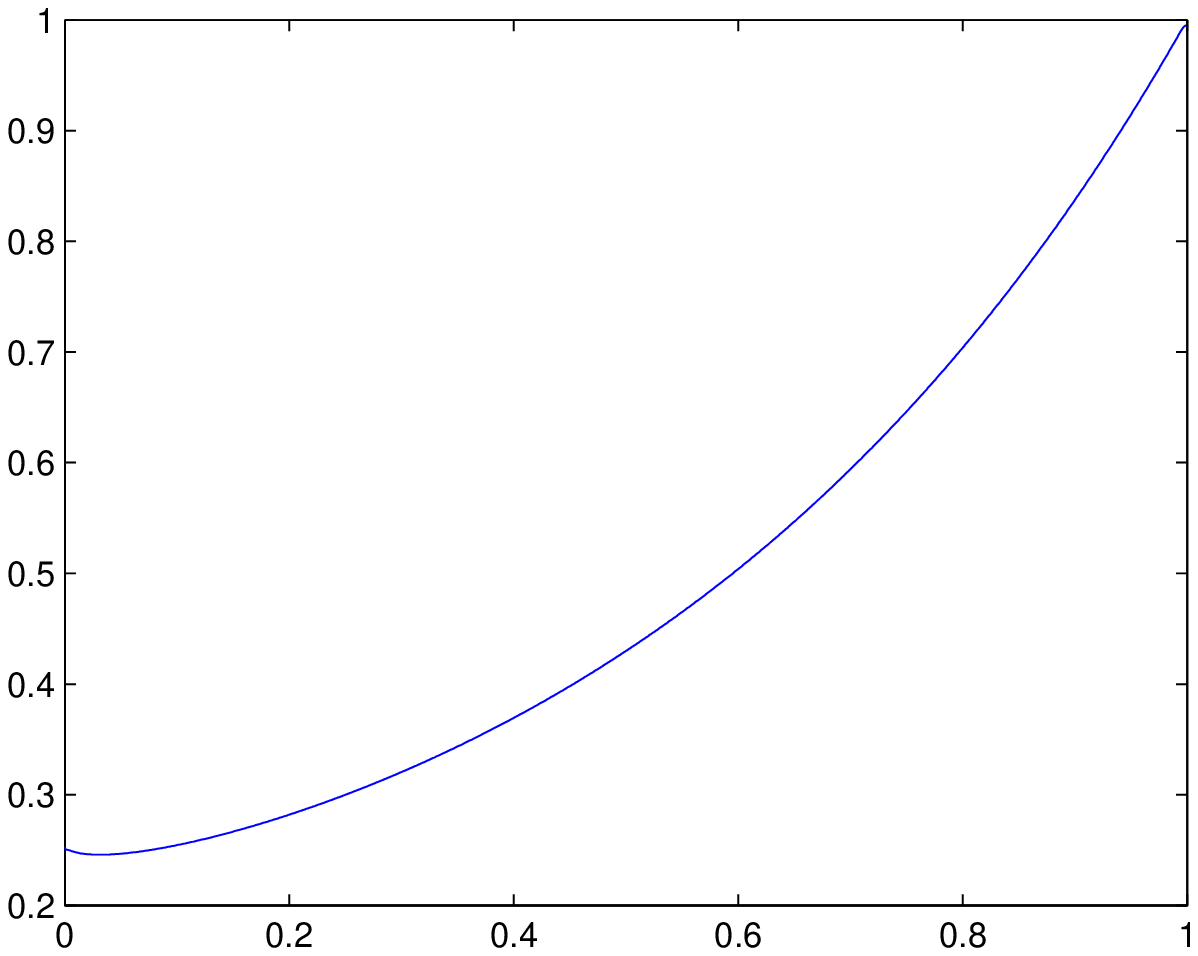,width=6.4cm,height=5cm,angle=0}
\psfig{file=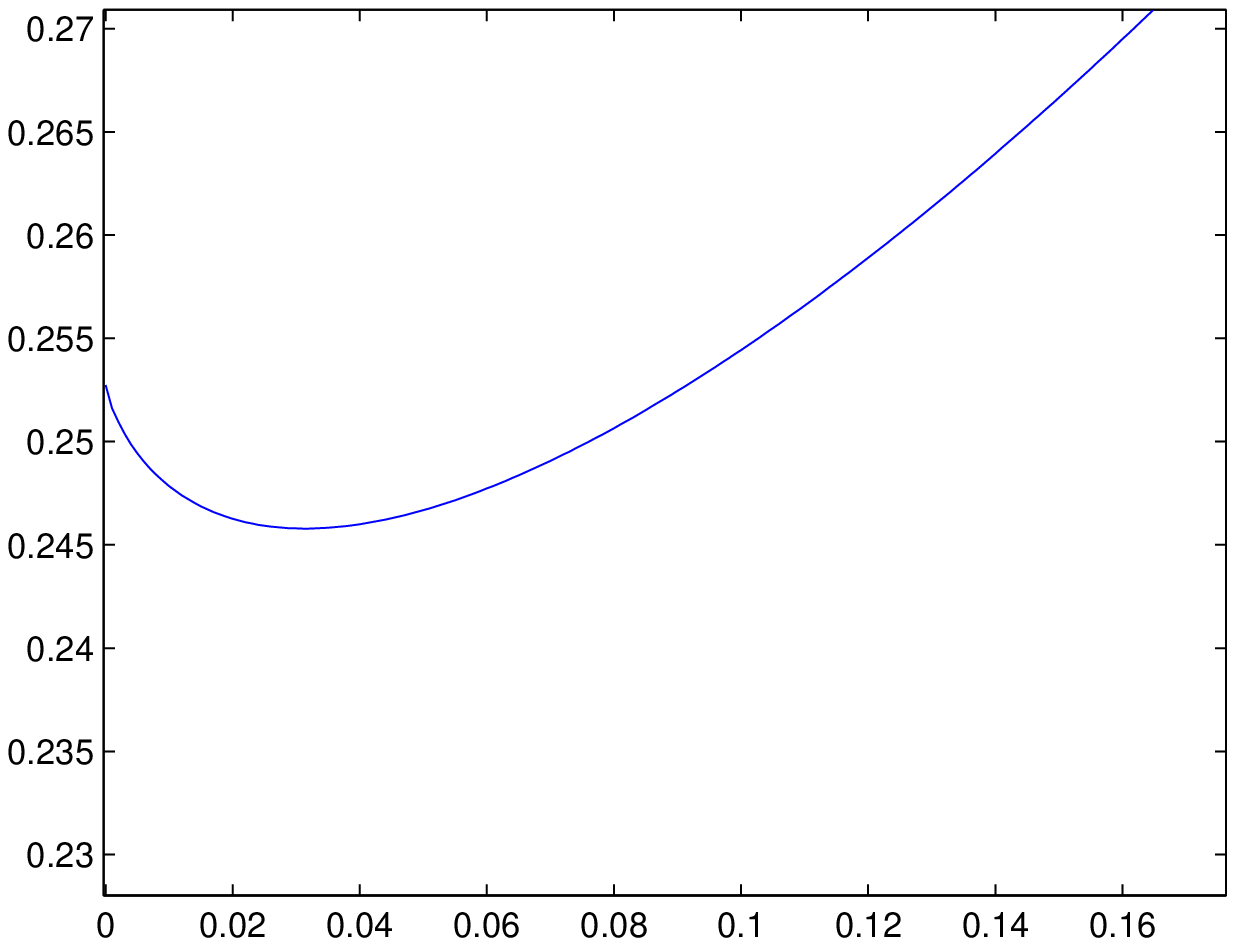,width=6.4cm,height=5cm,angle=0} }
\caption{Graph of $\mathrm{K}_3(\theta)$ on
the  interval $\theta\in[0,1]$ (left)with a closer look at its behavior
near $\theta=0$ (right).} \label{Fig:3D}
\end{figure}
\begin{proof}
We have to find the fundamental solution
$G_\lambda(k,l,m)$ of the equation
\begin{equation}\label{A(lambda)3D}
\mathbb{A}(\lambda)G_\lambda=
(\mathrm{D}_1^*\mathrm{D}_1+
\mathrm{D}_2^*\mathrm{D}_2+
\mathrm{D}_3^*\mathrm{D}_3+\lambda)G_\lambda=\delta.
\end{equation}
Similarly to the 1D and 2D cases we
find that the function
$$
\widehat{g}_\lambda(x,y,z):=\sum_{k,l,m=-\infty}^\infty
G_\lambda(k,l,m)e^{ikx+ily+imz},
$$
satisfies
$$
\widehat{g}_\lambda(x,y,z)=\frac{\frac14}{\frac\lambda 4+
\sin^2\frac x2+\sin^2\frac y2+\sin^2\frac z2}\,.
$$
As before we have the inequality
\begin{equation}\label{333}
u(0,0,0)^2\le G_\lambda(0,0,0)
(\|\nabla u\|^2+\lambda\|u\|^2),
\end{equation}
which saturates for $u=\mathrm{const}\cdot G_\lambda$.

For $\lambda>0$ as in the 1D and 2D cases we have
$\widehat{g}_\lambda\in L_2(\mathbb{T}^3)$, and, hence,
$G_\lambda\in l^2(\mathbb{Z}^3)$ for $\lambda>0$.
In particular, using \eqref{gradshtein} we find
\begin{equation}\label{G(0)3D}
\aligned
G_\lambda(0,0,0)=\frac1{8\pi^3}\int_{-\pi}^{\pi}\int_{-\pi}^{\pi}\int_{-\pi}^{\pi}
\widehat{g}_\lambda(x,y,z)dxdydz=\\=
\frac1{2\pi^2}
\int_0^\pi\frac{K\left(
\frac{1}{\frac\lambda 4+1+\sin^2\frac x2}\right)}
{\frac\lambda 4+1+\sin^2\frac x2}dx.
\endaligned
\end{equation}
However, unlike the previous two cases,
now
$\widehat{g}_\lambda$ is integrable for all
$\lambda\ge0$ including $\lambda=0$:
$\widehat{g}_\lambda\in L_1(\mathbb{T}^3)$  for
$\lambda\ge0$. Therefore
the Green's function $G_0$
is well defined and belongs to $l^\infty_0(\mathbb{Z}^3)$.
We point out, however, that since
$\widehat{g}_0\notin L_2(\mathbb{T}^3)$, it follows that
$G_0\notin l^2(\mathbb{Z}^3)$.

For $\lambda=0$,
the integrand has only  a logarithmic
singularity at  $x=0$ and
 we obtain
$$
\aligned
G_0(0,0,0)=
\frac1{2\pi^2}
\int_0^\pi\frac{K\left(
\frac{1}{1+\sin^2\frac x2}\right)}
{1+\sin^2\frac x2}dx.
\endaligned
$$
We now see that $f(\lambda):=G_\lambda(0,0,0)$
is continuous on $\lambda\in[0,\infty)$ and is of the order
$1/\lambda$ at infinity. This gives that
for $\theta\in(0,1)$ the function
$\lambda^\theta f(\lambda)$ vanishes both at the origin
and at infinity. Hence, it attains its maximum
at a (generically) unique point $\lambda_*(\theta)$,
and the claim of the theorem concerning the case
$\theta\in(0,1)$ follows in exactly the same way as in
Theorem~\ref{T:theta}.

Setting $\lambda=0$ in \eqref{333} we obtain~\eqref{3D0}
with~\eqref{K3}. It remains to verify that $\nabla G_0\in l^2(\mathbb{Z}^3)$.
To see this we use notation \eqref{fgh1} and Lemma~\ref{L:relations}.
We obtain
$$
\aligned
\|\nabla G_\lambda\|^2=h(\lambda)=f(\lambda)+\lambda f'(\lambda)=\\=
\frac1{8\pi^3}\int_{-\pi}^{\pi}\int_{-\pi}^{\pi}\int_{-\pi}^{\pi}
(\widehat{g}_\lambda(x,y,z)+\lambda\widehat{g}_\lambda(x,y,z)'_\lambda) dxdydz=\\=
\frac1{8\pi^3}\int_{-\pi}^{\pi}\int_{-\pi}^{\pi}\int_{-\pi}^{\pi}
\frac{\frac14(\sin^2\frac x2+\sin^2\frac y2+\sin^2\frac z2)}
{(\frac\lambda 4+\sin^2\frac x2+\sin^2\frac y2+\sin^2\frac z2)^2}dxdydz.
\endaligned
$$
Since the integral on right-hand side is bounded for $\lambda=0$
we have $\|\nabla G_0\|^2<\infty$. Finally,
$G_\lambda$ has strictly positive elements for $\lambda\ge0$,
since we have as before the maximum principle. In the case
when $\lambda=0$ we use, in addition, the fact that
$G_0\in l^\infty_0$.
The proof is complete.
\end{proof}

The graph of\ \ $\mathrm{K}_3(\theta)$ is shown in
Fig.~\ref{Fig:3D}.

\begin{remark}\label{R:d>3}
{\rm
Higher dimensional cases are treated similarly,
in particular, for $d\ge3$ and $\theta=0$
\begin{equation}\label{Kd(0)}
u(0)^2\le \mathrm{K}_d(0)
\|\nabla u\|^{2},\
\mathrm{K}_d(0)=
\frac1{(2\pi)^d}\int_{\mathrm{T}^d}
\frac{dx}{4(\sin^2\frac {x_1}2+\dots+\sin^2\frac {x_d}2)}\,.
\end{equation}
In \S\ref{S:Appl-Carlson} we give an independent elementary proof of
this inequality.
}
\end{remark}

\setcounter{equation}{0}
\section{Higher order difference operators}\label{S:2nd}

The method developed above admits  a straight forward generalization
to higher order difference operators. We consider the
second-order  operator in the one dimensional case:
\begin{equation}\label{1D2theta}
   u(0)^2\le \mathrm{K}_{1,2}(\theta)\|u\|^{2\theta}
   \|\Delta u\|^{2(1-\theta)},
\end{equation}
where
$$
-\Delta u(n):=\D^*\D u(n)=-\bigl(u(n+1)-2u(n)+u(n-1)\bigr).
$$
Accordingly, the operator $\mathbb{A}(\lambda)$ is
\begin{equation}\label{1D2Green}
\mathbb{A}(\lambda)=
\left\{
  \begin{array}{ll}
\phantom{-}\Delta^2+\lambda, & \hbox{for $\lambda>0$;} \\
-\Delta^2-\lambda, & \hbox{for $\lambda<-16$.}
  \end{array}
\right.
\end{equation}
Here
$$
\Delta^2u(n)=u(n+2)-4u(n+1)+6u(n)-4u(n-1)+u(n-2).
$$
As before, we have to find the Green's function $G_\lambda$ solving
 $\mathbb{A}(\lambda)G_\lambda=\delta$.
Furthermore,  for finding $\mathrm{K}_{1,2}(\theta)$ it suffices to
solve this equation  for $\lambda>0$. Setting
$$
\widehat{g}_\lambda(x):=\sum_{n=-\infty}^\infty
G_\lambda(n)e^{inx},
$$
and arguing as in Lemma~\ref{L:explicit} we get from
\eqref{1D2Green}
\begin{equation}\label{G4}
\aligned
1=\widehat{g}_\lambda(x)(\lambda+e^{-i2x}-4e^{-ix}
+6-4e^{ix}+e^{i2x})=\\=
\widehat{g}_\lambda(x)\bigl(\lambda+(e^{ix/2}-e^{-ix/2})^4\bigr)=
\widehat{g}_\lambda(x)\left(\lambda+16\sin^4\frac x2\right),
\endaligned
\end{equation}
so that
\begin{equation}\label{1D2Green(0)}
G_\lambda(0)=
\frac1{2\pi}\int_{-\pi}^\pi\frac{dx}{\lambda+16\sin^4\frac x2}=
\frac{\sqrt{2}}2\frac1{\lambda^{3/4}}
\sqrt{\frac{\sqrt{\lambda+16}+\sqrt{\lambda}}{\lambda+16}}\,.
\end{equation}
Now a word for word repetition of the argument in
Theorem~\ref{T:theta} gives that
$$
\mathrm{K}_{1,2}(\theta)=
\frac{1}{\theta^\theta(1-\theta)^{1-\theta}}\cdot
\sup_{\lambda>0}\lambda^{\theta} G_\lambda(0).
$$
Therefore we see from~\eqref{1D2Green(0)} that
$\mathrm{K}_{1,2}(\theta)<\infty$ if and only if
$$
\frac34\le\theta\le1.
$$
For example, for $\theta=3/4$ supremum is the maximum attained
at $\lambda_*(3/4)=16/3$, giving
$$
\mathrm{K}_{1,2}(3/4)=
\frac{4}{3^{3/4}}\cdot
\lambda^{3/4} G_\lambda(0)\vert_{\lambda_*(3/4)=\frac{16}3}=
\frac{\sqrt{2}}2.
$$
We only mention that in the general case
\begin{equation}\label{Maple}
\aligned
\lambda_*(\theta)=
\mathrm{argmax}_{\lambda>0}
\,\lambda^{\theta-3/4}
\sqrt{\frac{\sqrt{\lambda+16}+\sqrt{\lambda}}{\lambda+16}}
=\\
=\frac{64\theta-32\theta^2-29+
\sqrt{32\theta-23}}
{2\theta^2-5\theta+3},
\endaligned
\end{equation}
however, the corresponding substitution produces
a  long (but explicit) formula for $\mathrm{K}_{1,2}(\theta)$,
 and instead we present in Fig.\ref{fig:K2(theta12)}
 the graph
of the sharp constant $\mathrm{K}_{1,2}(\theta)$,
where $\mathrm{K}_{1,2}(3/4)=\sqrt{2}/2$ and
$\mathrm{K}_{1,2}(1)=1$.
\begin{figure}[htb]
\centerline{\psfig{file=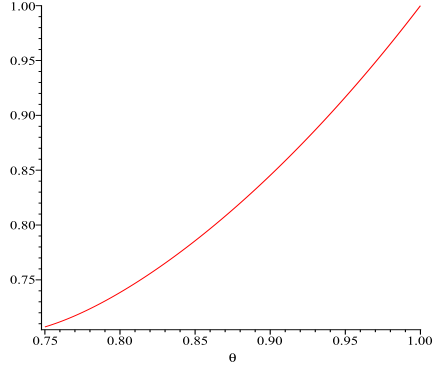,width=6cm,height=5cm,angle=0}}
\caption{Graph of $\mathrm{K}_{1,2}(\theta)$ on $\theta\in [3/4,1]$.}
\label{fig:K2(theta12)}
\end{figure}

Finally, it is possible to find $G_\lambda(n)$ explicitly. In fact,
the free recurrence relation $ \Delta^2G_\lambda+\lambda G_\lambda=0$
has the characteristic equation
$$
q^2-4q+(6+\lambda)-4q^{-1}+q^{-2}=0,
$$
or $(q^{1/2}-q^{-1/2})^4=-\lambda$, which decomposes into two quadratic
equations
$$
q+\frac 1q-2=i\sqrt{\lambda}\quad\text{and}\quad q+\frac1q-2=-i\sqrt{\lambda},
$$
with four roots $q_1,q_2,q_3,q_4$, where $q_2=1/q_1$, $q_3=\bar q_2$, $q_4=\bar q_1$,
where
\begin{equation}\label{q(lambda)}
q_1=q(\lambda)=1-\frac{\lambda^{1/4}\sqrt{\sqrt{\lambda+16}-\sqrt{\lambda}}}
{2\sqrt{2}}+i\left(\frac{\sqrt{\lambda}}2-
\frac{\lambda^{1/4}\sqrt{\sqrt{\lambda+16}+\sqrt{\lambda}}}
{2\sqrt{2}}\right).
\end{equation}
Since $|q(\lambda)|<1$ for $\lambda>0$, it follows that any symmetric $l_2$-solution
of \eqref{1D2Green} is of the form $a(\lambda)q(\lambda)^{|n|}+b(\lambda)\bar q(\lambda)^{|n|}$,
and since, in addition $G_\lambda(n)$ is real, we have
$$
G_\lambda(n)=a(\lambda)q(\lambda)^{|n|}+\bar a(\lambda)\bar q(\lambda)^{|n|}.
$$
Setting $n=0$ and $n=1$ we obtain a linear system for $a(\lambda)$
$$
\aligned
&G_\lambda(0)=2\Re a(\lambda)\\
&G_\lambda(1)=a(\lambda)q(\lambda)+\bar a(\lambda)\bar q(\lambda),
\endaligned
$$
where $G_\lambda(0)$ is given in \eqref{1D2Green(0)} and
$$
G_\lambda(1)=
\frac1{\pi}\int_{0}^\pi\frac{\cos x\,dx}{\lambda+16\sin^4\frac x2}=
\frac{\sqrt{2}}2\frac1{\lambda^{3/4}}
\frac{\sqrt{\lambda+16}-\sqrt{\lambda}}{\sqrt{\sqrt{\lambda+16}+\sqrt{\lambda}}
\,\sqrt{\lambda+16}}\,.
$$
Solving this system we find $a(\lambda)$:
$$
a(\lambda)=
\frac{\sqrt{2}}4\frac1
{{\lambda^{3/4}}\sqrt{\lambda+16}\sqrt{\sqrt{\lambda+16}+\sqrt{\lambda}}}
\left( \sqrt{\lambda+16}+\sqrt{\lambda} +4i\right),
$$
and, consequently, the formula for $G_\lambda(n)$ with $\lambda>0$:
\begin{equation}\label{expl2ord}
\aligned
G_\lambda(n)=\frac1{\pi}\int_{0}^\pi\frac{\cos nx\,dx}{\lambda+16\sin^4\frac x2}=
\frac{\Re\left[\bigl(\sqrt{\lambda+16}+\sqrt{\lambda} +4i\bigr)\cdot q(\lambda)^{|n|}
\right]}
{\sqrt{2}\,{\lambda^{3/4}}\sqrt{\lambda+16}\sqrt{\sqrt{\lambda+16}+\sqrt{\lambda}}},
\endaligned
\end{equation}
where $q(\lambda)$ is given in~\eqref{q(lambda)}.

\begin{figure}[htb]
\centerline{\psfig{file=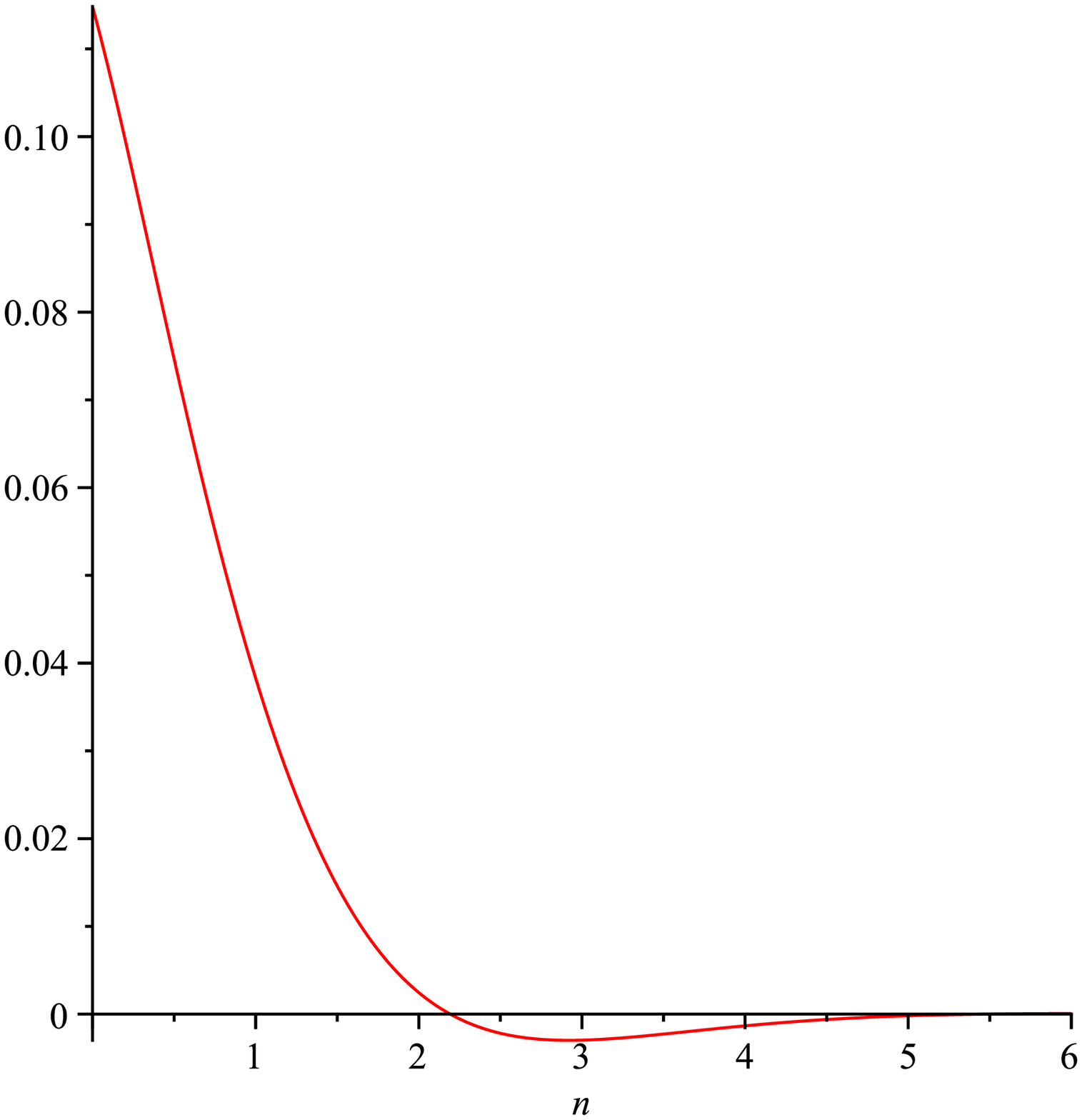,width=6.4cm,height=5cm,angle=0}
\psfig{file=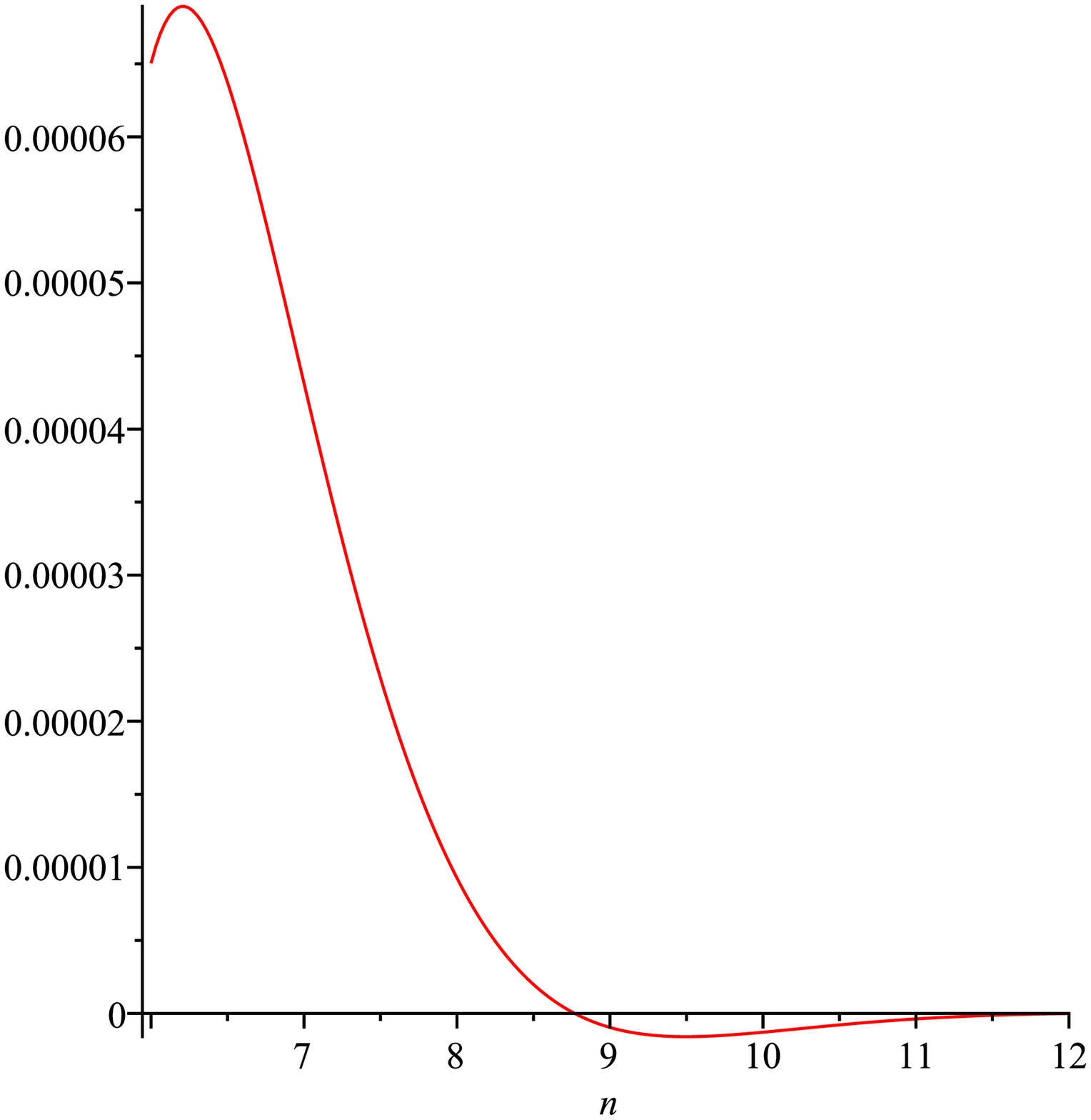,width=6.4cm,height=5cm,angle=0} }
\caption{Graph of the maximizer $G_\frac{16}3(n)$ for
$n\in[0,6]$ (left) and $n\in[6,12]$ (right)} \label{Fig:1D2ndord}
\end{figure}
\medskip

Thus, we obtain the following result.

\begin{theorem}\label{T:theta12}
Inequality~\eqref{1D2theta} holds for $\theta\in[3/4,1]$.
In particular, in the limiting case $\theta=3/4$
\begin{equation}\label{1D2theta3/4}
   u(0)^2\le \frac{\sqrt{2}}2\,\|u\|^{3/2}
   \|\Delta u\|^{1/2}.
\end{equation}
In the general case,
$$
\mathrm{K}_{1,2}(\theta)=
\frac{1}{\theta^\theta(1-\theta)^{1-\theta}}\cdot
\lambda_*(\theta)^{\theta} G_{\lambda_*(\theta)}(0),
$$
where $\lambda_*(\theta)$ is given in~\eqref{Maple}
and $G_\lambda(0)$ in \eqref{1D2Green(0)} (see also~\eqref{expl2ord}).
For $\theta\in[3/4,1)$ the unique extremal is
$u_{\lambda_*(\theta)}=G_{\lambda_*(\theta)}$.
For $\theta=1$, $\lambda_*(1)=\infty$ and $u_*=\delta$.
\end{theorem}

\begin{remark}\label{R:<0}
{\rm
It is not difficult to find the function $\mathbb{V}(d)$, that is,  the
solution of the maximization problem
\begin{equation}\label{V(d)2ord}
\mathbb{V}(d):=\sup\bigl\{u(0)^2:\
u\in l^2(\mathbb{Z}), \ \|u\|^2=1,\
\|\Delta u\|^2=d\bigr\},
\end{equation}
where $d\in[0,16]$. For this purpose we also need the expression for the Green's
function $G_\lambda(0)$ in the region $\lambda\le-16$, which is
as follows
\begin{equation}\label{1D2Green(0)neg}
G_\lambda(0)=
-\frac1{2\pi}\int_{-\pi}^\pi\frac{dx}{\lambda+16\sin^4\frac x2}=
\frac1{2(-\lambda)^{3/4}}
{\frac{\sqrt{\sqrt{-\lambda}+4}+\sqrt{\sqrt{-\lambda}-4}}{\sqrt{-\lambda-16}}}\,.
\end{equation}
Using \eqref{1D2Green(0)}, \eqref{1D2Green(0)neg} we can
write down a parametric representation of $\mathbb{V}(d)$ as in  Theorem~\ref{T:V(d)2}, but instead we merely show its graph in Fig.\ref{F:V(d)non-sym}.
\begin{figure}[htb]
\centerline{\psfig{file=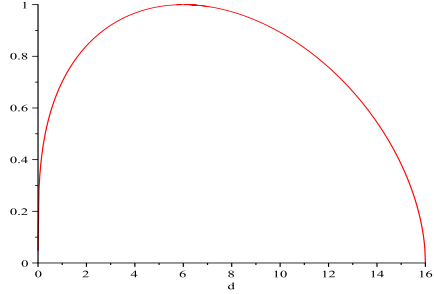,width=6cm,height=4cm,angle=0}
}
\caption{Graph of  $\mathbb{V}(d)$ defined in~\eqref{V(d)2ord}.
}
 \label{F:V(d)non-sym}
\end{figure}

This time we do not have the maximum principle, and the Green's function
 $G_\lambda(n)$ is not  positive for all $n$, but is rather oscillating with exponentially
 decaying amplitude, see Fig.~\ref{Fig:1D2ndord}. Nor do we have the
symmetry $\mathbb{V}(d)=\mathbb{V}(16-d)$ in Fig.~\ref{F:V(d)non-sym}
that we have seen in the first-order inequalities in the
one- and two-dimensional cases, see~\eqref{Vsymm} and  \eqref{Vsymm2}.
The maximum is attained at $d=6$ corresponding to $u=\delta$. The component $\lambda\in(0,\infty)$ of the resolvent set
  corresponds to $d\in (0,6)$
and  $\lambda\in(-\infty,-16)$  corresponds to $d\in (6,16)$.
 }
\end{remark}

It is worth to compare the results so obtained in the
discrete case with the corresponding interpolation
inequalities for Sobolev spaces in the continuous case.
It is well known that the interpolation inequality on
the whole line $\mathbb{R}$
\begin{equation}\label{1donR}
\|f\|^2_{L_\infty}\le \mathrm{C}_{1,n}(\theta)
\|f\|^{2\theta}\|f^{(n)}\|^{2(1-\theta)},
\end{equation}
where $f\in H^n(\mathbb{R})$, holds only for $\theta=1-\tfrac1{2n}$. The sharp constant was found in
\cite{Taikov}:
\begin{equation}\label{Taikov-const}
\mathrm{C}_{1,n}(\theta)=\frac{1}{\theta^\theta(1-\theta)^{1-\theta}}
\frac1\pi\int_0^\infty\frac{dx}{1+x^{2n}}=
\frac{1}{\theta^\theta(1-\theta)^{1-\theta}}
\frac1{2n}\frac1{\sin\frac\pi{2n}}\,.
\end{equation}

Thus, for first-order inequalities both in the discrete and continuous cases
the constants are equal to $1$, while for the second-order
inequalities we see from~\eqref{1D2theta3/4} and
\eqref{Taikov-const} that
$$
\mathrm{C}_{1,2}(3/4)=\left(\frac4{27}\right)^{1/4}
<\frac{\sqrt{2}}2=\mathrm{K}_{1,2}(3/4).
$$

The next theorem states that for higher order
inequalities the constants in the discrete case are always
strictly greater than those in the continuous case.
\begin{theorem}\label{T:theta1n}
Let $n\ge1$ and let $u\in l^2(\mathbb{Z})$. The inequality
\begin{equation}\label{theta1n}
 u(0)^2\le \mathrm{K}_{1,n}(\theta)
 \|u\|^{2\theta}
   \|\mathcal{D}^nu\|^{2(1-\theta)},
\end{equation}
where
\begin{equation}\label{Dn}
\mathcal{D}^n:=\left\{
\begin{array}{ll}
 \Delta^{n/2}, & \hbox{for $n$  even,} \\
\nabla  \Delta^{(n-1)/2}, & \hbox{for $n$ odd,}
\end{array}
 \right.
\end{equation}
holds for $\theta\in[1-1/(2n),1]$ and
\begin{equation}\label{K1n}
\mathrm{K}_{1,n}(\theta)=
\frac{1}{\theta^\theta(1-\theta)^{1-\theta}}\frac 1\pi
\sup_{\lambda>0}\lambda^{\theta}
\int_0^\pi\frac{dx}{\lambda+2^{2n}\sin^{2n}\frac x2}\,.
\end{equation}
For all $\theta\in[1-1/(2n),1]$ supremum is the maximum.
If $n\ge2$, then for $\theta=\theta_*:=1-1/(2n)$
the constants in the continuous and discrete inequalities
satisfy
\begin{equation}\label{KandK}
\mathrm{C}_{1,n}(\theta_*)<\mathrm{K}_{1,n}(\theta_*).
\end{equation}
\end{theorem}
\begin{proof} Following the scheme developed above
we look for the solution of the equation
$$
(-1)^{n}\Delta^nG_\lambda+\lambda G_\lambda=\delta,
$$
and as in~\eqref{G4} find that
$$
G_\lambda(0)=
\frac1{\pi}\int_{0}^\pi
\frac{dx}{\lambda+2^{2n}\sin^{2n}\frac x2},
$$
which proves~\eqref{K1n} (whenever the supremum is finite).
Using $\sin^2\frac x2=\tan^2\frac x2/(1+\tan^2\frac x2)$
and changing the variable $\tan\frac x2=\sqrt{\mu}t$, where $\mu=\lambda^{1/n}$ we have
$$
\lambda^{\theta_*}\int_0^\pi
\frac{dx}{\lambda+2^{2n}\sin^{2n}\frac x2}=
\int_0^\infty\frac{2dt}
{(1+\mu t^2)(1+\frac{2^{2n}t^{2n}}{(1+\mu t^2)^n})}
=: S(\mu).
$$
Clearly $S(\infty)=0$, and we have to study $S(\mu)$ as $\mu\to0$.
The integral converges uniformly for $\mu\in[0,1]$,
since the denominator is greater then $1$ for $t\in[0,1]$ and is greater then
$1+c(n)t^2$ for $t\ge1$ observing that
$t^{2n}/(1+\mu t^2)^{n-1}\ge c_1(n) t^2$ uniformly for $\mu\in[0,1]$.
Therefore
$$
\lim_{\mu\to 0}S(\mu)=S(0)=\int_0^\infty\frac{2dt}
{1+2^{2n}t^{2n}}=\int_0^\infty\frac{dx}{1+x^{2n}}\,,
$$
which proves, in the first place,  that the right-hand side in~\eqref{K1n} is finite if and only if
$\theta\in[\theta_*,1]$ and, secondly, that non-strict inequality
\eqref{KandK} holds.
Finally, for $n\ge2$ we have strict inequality since
$$
S'_\mu(0)=2\int_0^\infty\left[\frac{n2^{2n}t^{2n+2}}{(1+2^{2n}t^{2n})^2}
-\frac{t^2}{1+2^{2n}t^{2n}}\right]dt=
\frac1{16n}\frac\pi{\sin\frac{3\pi}{2n}}>0.
$$
For $n=1$ we have $\lambda=\mu$,  $S'_\mu(0)<0$ and
$$
S(\mu)=\frac\pi{\sqrt{\mu+4}}
$$
is strictly decreasing not only at $\mu=0$ but for all $\mu\ge0$,  the fact
that we have already seen in~\eqref{G(0)}.
\end{proof}
\begin{remark}\label{R:Per}
{\rm
Inequality \eqref{theta1n} holds for $\theta\in[1-1/(2n),1]$,
that is, when the weight of the stronger norm, which is the $l_2$-norm,
increases. Accordingly, inequality \eqref{1donR}
for periodic functions with mean value zero holds for
$\theta$ in the {\it complementary} interval $\theta\in [0, 1-1/(2n)]$,
when the weight of the stronger norm, which is the
$L_2$-norm of the $n$-th derivative, increases:
\begin{equation}\label{1donRper}
\|f\|^2_{L_\infty}\le \mathrm{C}_{1,n}^{\mathrm{per}}(\theta)
\|f\|^{2\theta}\|f^{(n)}\|^{2(1-\theta)},\qquad
\theta\in[0,1-\tfrac1{2n}],
\end{equation}
where
$$
f\in H^n_{\mathrm{per}}(\mathbb{S}^1),\qquad
\int_0^{2\pi}f(x)dx=0.
$$
A general method for finding sharp constants in
interpolation inequalities of $L_\infty$-$L_2$-$L_2$-type
was developed in \cite{I98JLMS}, \cite{Zelik}, \cite{IZ},
which was also used in the discrete case in the present paper.
For example, for $n=1$
\begin{equation}\label{C11theta}
\mathrm{C}_{1,1}^{\mathrm{per}}(\theta)=
\frac1{\theta^\theta(1-\theta)^{1-\theta}}
\sup_{\lambda\ge0}\lambda^\theta G(\lambda),
\qquad\theta\in[0,1/2],
\end{equation}
where $G_\lambda(x,\xi)=\frac1{2\pi}\sum_{x\in\mathbb{Z}\setminus\{0\}}
\frac{e^{ik(x-\xi)}}{k^2+\lambda}$ is the Green's function
of the equation
$$
-G_\lambda(x,\xi)''+\lambda G_\lambda(x,\xi)=\delta(x-\xi),
\quad x,\xi\in[0,2\pi]^\mathrm{per},
$$
and
$$
G(\lambda):=G_\lambda(\xi,\xi)=
\frac1\pi\sum_{k=1}^\infty\frac1{k^2+\lambda}=
\frac1{2\pi}\frac{\pi\sqrt\lambda\coth(\pi\sqrt\lambda)-1}\lambda.
$$

For the limiting $\theta=\theta_*$=1/2 the constant is
the same as on $\mathbb{R}$: $\mathrm{C}_{1,1}^{\mathrm{per}}(\theta_*)=
\mathrm{C}_{1,1}(\theta_*)=1$. The graph of $\mathrm{C}_{1,1}^{\mathrm{per}}(\theta)$
on the interval $\theta\in[0,1/2]$ is shown in Fig.\ref{fig:K(theta)} on the left.
Observe that $\mathrm{C}_{1,1}(0)=\pi/6$.
 }
\end{remark}

\setcounter{equation}{0}
\section{Applications}
\label{S:Appl-Carlson}
\subsection*{Discrete and integral Carlson inequalities}\label{SS:Carl}
 We now discuss applications of the inequalities for the discrete operators,
 and our first group of results concerns Carlson inequalities.
  The original Carlson inequality~\cite{Carl} is as follows:
\begin{equation}\label{Carlson}
\left(\sum_{k=1}^\infty a_k\right)^2\le
\pi\left(\sum_{k=1}^\infty a_k^2\right)^{1/2}
\left(\sum_{k=1}^\infty k^2a_k^2\right)^{1/2},
\end{equation}
where the constant $\pi$ is sharp and cannot be attained at a non
identically zero sequence $\{a_k\}_{k=1}^\infty$. This inequality has attracted
a lot of interest and has been a source of generalizations and improvements
(see, for example, \cite{HLP}, \cite{Carl_book}
and the references therein, and also~\cite{IZ} for the most recent strengthening of
\eqref{Carlson}).
Inequality \eqref {Carlson} has an integral analog
(with the same sharp constant)
\begin{equation}\label{Carlson-int}
\left(\int_0^\infty g(t)dt\right)^2\le
\pi\left(\int_0^\infty g(t)^2dt\right)^{1/2}
\left(\int_0^\infty t^2g(t)^2dt\right)^{1/2}.
\end{equation}
As was first observed in~\cite{Hardy}, inequality \eqref{Carlson} is equivalent
to the inequality
\begin{equation}\label{1dSob}
\|f\|_\infty^2\le 1\cdot\|f\|\|f'\|,
\end{equation}
for periodic functions $f\in H^1_{\mathrm{per}}(0,2\pi)$, $\int_0^{2\pi}f(x)dx=0$,
by setting for a  sequence $\{a_k\}_{k=1}^\infty$
$$
f(x)=\sum_{k=-\infty}^\infty a_{|k|}e^{ikx},\qquad a_0=0.
$$

Accordingly, inequality~\eqref{1dSob} for $f\in H^1(\mathbb{R})$
is equivalent (as was first probably  observed in~\cite{Nagy})
to~\eqref{Carlson-int} by setting $g=\mathcal{F}f$ and further restricting $g$
(and $f$) to even functions. Furthermore, the unique
(up to scaling) extremal function $f_*(x)=e^{-|x|}$ in~\eqref{1dSob} on
the whole axis produces the extremal function
$g_*(t)=1/(1+t^2)$ in~\eqref{Carlson-int}.

In the similar way, discrete inequalities have equivalent
integral analogs.
Let $\mathscr{F}$ be the discrete Fourier
transform $\mathscr{F}:\{a(n)\}\to \widehat{a}(x)$, where
$$
 \widehat{a}(x)=\sum_{n\in\mathbb{Z}^d}a(n)e^{inx},
\qquad a(n)=(2\pi)^{-d}\int_{\mathbb{T}^d}\widehat{a}(x)e^{-inx}dx.
$$
Then for $e_j=(0,\dots,0,1,0,\dots,0)$ with $1$ on the $j$th place
$$
\aligned
\mathrm{D}_ja(n)=a(n+e_j)-a(n)=(2\pi)^{-d}\int_{\mathbb{T}^d}
\widehat{a}(x)(e^{-i(n+e_j)x}-e^{-inx})dx=\\=
(2\pi)^{-d}\int_{\mathbb{T}^d} \widehat{a}(x)e^{-ix/2}(-2i)
\sin\tfrac {x_j}2\,e^{-inx}dx.
\endaligned
$$
Therefore
\begin{equation}\label{Parseval}
\|\mathrm{D}_ja\|^2=(2\pi)^{-d}\int_{\mathbb{T}^d} |\widehat{a}(x)|^24\sin^2\tfrac {x_j}2\,dx.
\end{equation}
and, finally,
\begin{equation}\label{Parsevald}
\|a\|^2=
(2\pi)^{-d}\|\widehat{a}\|^2,\
\|\nabla a\|^2=(2\pi)^{-d}\int_{\mathbb{T}^d}|\widehat{a}(x)|^24
\sum_{j=1}^d\nolimits\sin^2\tfrac {x_j}2\,dx.
\end{equation}
Thus, we have proved the following result.
\begin{theorem}\label{T:tintCarlon}
Let $1/2<\theta\le1$.  The inequality
$$
 u(0)^2\le \mathrm{K}_1(\theta)\|u\|^{2\theta}
   \|\D u\|^{2(1-\theta)},\qquad u\in l_2(\mathbb{Z})
$$
established in Theorem~\ref{T:theta} is equivalent to
the inequality
\begin{equation}\label{Carlson-int2pi}
\left(\int_0^{2\pi} g(x)dx\right)^2\le
2\pi\mathrm{K}_1(\theta)\left(\int_0^{2\pi} g(x)^2dx\right)^\theta
\left(\int_0^{2\pi} 4\sin^2\tfrac x2\,g(x)^2dx\right)^{1-\theta},
\end{equation}
for $g\in L_2(0,2\pi)$. Here
$\mathrm{K}_1(\theta)=\frac12\left(2/\theta\right)^\theta
(2\theta-1)^{\theta-1/2}$ (see~\eqref{K(theta)}).
In the limiting case inequality~\eqref{Lap1/2}
is equivalent to
\begin{equation}\label{Carlson-limit}
\aligned
&\left(\int_0^{2\pi} g(x)dx\right)^2\le\\
\pi\sqrt{4-\frac{\int_0^{2\pi} 4\sin^2\tfrac x2\,g(x)^2dx}{\int_0^{2\pi} g(x)^2dx}
}
&\left(\int_0^{2\pi} g(x)^2dx\right)^{1/2}
\left(\int_0^{2\pi} 4\sin^2\tfrac x2\,g(x)^2dx\right)^{1/2}
\endaligned
\end{equation}
\end{theorem}
\begin{proof} The proof follows from Theorem~\ref{T:theta}
and~\eqref{Parseval}. We also point out that
for $\theta\in(1/2,1)$ inequality~\eqref{Carlson-int2pi}
saturates for
$$
{g}_{\lambda_*}(x)=
\frac1{\lambda_*+4\sin^2\frac x2}\,,\qquad
\lambda_*=\lambda_*(\theta)=\frac{4\theta-2}{1-\theta}\,.
$$
for $\theta=1/2$ no extremals exist and maximizing sequence
is obtained by letting $\lambda_*\to0$;
finally for $\theta=1$, ~\eqref{Carlson-int2pi}
saturates at constants.

For each $d$, $0<d<4$  and $\lambda(d)=\frac{2d}{{2-d}}$ inequality \eqref{Carlson-limit} saturates at
$$
{g}_{\lambda(d)}(x)=
\frac1{\lambda(d)+4\sin^2\frac x2}\,,\
\text{with}\
\frac{\int_0^{2\pi} 4\sin^2\tfrac x2\,g_{\lambda(d)}(x)^2dx}{\int_0^{2\pi} g_{\lambda(d)}(x)^2dx}=d.
$$
For $d=2$, ${g}=\mathrm{const}$.
\end{proof}
\begin{remark}\label{R:Carlson1ndord}
{\rm
Corresponding to~\eqref{1D2theta3/4} is the integral inequality
\begin{equation}\label{Carlson-int2pi2ndord}
\left(\int_0^{2\pi} g(x)dx\right)^2\le
\pi\sqrt{2}\left(\int_0^{2\pi} g(x)^2dx\right)^{3/4}
\left(\int_0^{2\pi} 16\sin^4\tfrac x2\,g(x)^2dx\right)^{1/4},
\end{equation}
which turns into equality for
$$
g_*(x)=\frac1{\frac{16}3+16\sin^4\frac x2}\,.
$$
}
\end{remark}
\begin{remark}\label{R:Carlson2ndord}
{\rm
The integral analog of the two dimensional discrete inequality is
\begin{equation}\label{Carlson-int2pi2ndord2d}
\aligned
&\left(\int_{\mathbb{T}^2} g(x,y)dxdy\right)^2\le\\
(2\pi)^{2}\mathrm{K}_2(\theta)&\left(\int_{\mathbb{T}^2}
 g(x,y)^2dxdy\right)^{\theta}
\left(\int_{\mathbb{T}^2} 4(\sin^2\tfrac x2+\sin^2\tfrac y2)
\,g(x,y)^2dxdy\right)^{1-\theta},
\endaligned
\end{equation}
where $\theta\in(0,1]$, and $\mathrm{K}_2(\theta)$
is defined in~\eqref{K(theta)2}.
}
\end{remark}
\begin{remark}\label{R:Carlson3ndord}
{\rm
In the $d$-dimensional case, $d\ge3$,  for the  exponent
$\theta=0$ the Parceval's identities \eqref{Parsevald}
provide an independent {\it elementary}
 proof of \eqref{Kd(0)}. In fact, setting
$g_0(x)=4\sum_{j=1}^d\nolimits\sin^2\tfrac {x_j}2$ we have
 \begin{equation}\label{Carl3d}
 \aligned
 (2\pi)^{2d}|a(0)|^2=
\left|\int_{\mathbb{T}^d}\widehat{a}(x)dx\right|^2\le\\
\le
\left(\int_{\mathbb{T}^d}|\widehat{a}(x)|g_0(x)^{1/2}g_0(x)^{-1/2}dx\right)^2\le\\
\le
\int_{\mathbb{T}^d}|\widehat{a}(x)|^2g_0(x)dx
\int_{\mathbb{T}^d}g_0(x)^{-1}dx=
(2\pi)^{2d}\mathrm{K}_d(0)\|\nabla a\|^2,
\endaligned
\end{equation}
which proves \eqref{Kd(0)}.
 }
\end{remark}

This  approach
can be generalized to the $l^p$-case for the proof of the discrete Sobolev
type inequality in the {\it non-limiting} case~\eqref{Sobolev}.
Here in addition to
 the Parseval's identity we also use the Hausdorff-Young inequality (see, for instance,
\cite{B-L}):
\begin{equation}\label{H-Y}
\|\widehat a\|_{L_p(\mathbb{T}^d)}\le (2\pi)^{d/p}\|a\|_{l^{p'}(\mathbb{Z}^d)},
\end{equation}
where $p\ge2$ and $p'=p/(p-1)$.

In fact, we have $\|\mathscr{F}\|_{l^2\to L_2}=(2\pi)^{d/2}$
and $\|\mathscr{F}\|_{l^1\to L_\infty}=1$ and by the Riesz--Thorin
interpolation theorem
$$
\|\mathscr{F}\|_{l^{p'}\to L_p}\le (2\pi)^{d\theta/2}=
(2\pi)^{d/p},
$$
where $\frac1{p'}=\frac\theta 2+\frac{1-\theta}1$,
$\frac1{p}=\frac\theta 2+\frac{1-\theta}\infty$. We also observe that
\eqref{H-Y} becomes an equality for $\widehat{a}(x)=1$ and $a=\delta$.

 Setting $q=2p$ in \eqref{Sobolev},
$v(n):=\overline{u(n)^p}u(n)^{p-1}$, and using the auxiliary
inequality  \eqref{auxv}, \eqref{Ipd} below, we obtain
$$
\aligned
\|u\|_{l^{2p}}^{2p}=\sum_{n\in\mathbb{Z}^d}v(n)u(n)=
\sum_{n\in\mathbb{Z}^d}(\mathrm{D}^*\mathrm{D}\Delta^{-1}v(n))u(n)=\\=
\sum_{n\in\mathbb{Z}^d}\mathrm{D}\Delta^{-1}v(n)\mathrm{D}u(n)\le
\left(\sum_{n\in\mathbb{Z}^d}|\mathrm{D}\Delta^{-1}v(n)|^{2}\right)^{1/2}\|\mathrm{D}u\|\le\\
\le\left(\frac14(2\pi)^{d/p'}I_{p',d}\right)^{1/2}
\|v\|_{l^{(2p)'}}\|\mathrm{D}u\|=
\left(\frac14(2\pi)^{d/p'}I_{p',d}\right)^{1/2}
\|u\|_{l^{2p}}^{2p-1}\|\mathrm{D}u\|.
\endaligned
$$
It remains to prove \eqref{auxv}. By  H\"older's inequality and \eqref{H-Y}
we have
\begin{equation}\label{auxv}
\aligned
\sum_{n\in\mathbb{Z}^d}|\mathrm{D}\Delta^{-1}v(n)|^2=\frac{(2\pi)^d}{4}
\int_{\mathrm{T}^d}\frac{|\widehat{v}(x)|^2dx}{\sum_{j=1}^d\nolimits\sin^2\tfrac {x_j}2}\le\\
\le
\frac{(2\pi)^d}{4}
I_{p',d}
\left(\int_{\mathrm{T}^d}{|\widehat{v}(x)|^{2p}dx}\right)^{1/p}\le\\
\le
\frac{(2\pi)^d}{4}
I_{p',d}(2\pi)^{d/p}
\|v\|_{l^{(2p)'}}^2=\frac14(2\pi)^{d(p+1)/p}I_{p',d}
\|v\|_{l^{(2p)'}}^2
\endaligned
\end{equation}
where
\begin{equation}\label{Ipd}
I_{p',d}=
\left(\int_{\mathrm{T}^d}\frac{dx}{(\sum_{j=1}^d\nolimits\sin^2\tfrac {x_j}2)^{p'}}\right)^{1/p'}\!<\infty
\ \text{for}\ p'<d/2\Leftrightarrow p>d/(d-2).
\end{equation}

Thus we obtain the following result.
\begin{theorem}\label{T:Lapmultdiv}
Let $d\ge3$ and $2p>2d/(d-2)$. Then
$$
\|u\|^2_{l^{2p}(\mathbb{Z})^d}=\left(\sum_{n\in\mathrm{Z}^d}|u(n)|^{2p}\right)^{1/p}\le
\frac14(2\pi)^{d(p+1)/p}I_{p',d}\|\mathrm{D}u\|^2,
$$
where $I_{p',d}$ is defined in \eqref{Ipd}.
\end{theorem}
\begin{remark}\label{R:Sobolev}
{\rm
We do not claim that the constant here is sharp.
Moreover, it blows up for $2p=2d/(d-2)$, while it can be
shown that  the inequality
still holds. However, the constant is sharp in the
opposite limit $p=\infty$, see \eqref{Kd(0)}.
}
\end{remark}

\subsection*{Spectral inequalities for discrete operators}\label{SS:Spectral}

Interpolation inequalities characterizing imbeddings of Sobolev spaces
into the space of bounded continuous functions have important
applications in spectral theory. The original fruitful idea in~\cite{E-F}
has been generalized in~\cite{D-L-L} to give best-known
estimates for the Lieb--Thirring constants in estimates for the negative
trace of Schr\"{o}dinger operators.

In this section we apply our sharp interpolation inequalities with
the method of~\cite{E-F} for estimates of the negative trace of the
discrete operators \cite{Arman}.

We write the inequalities  obtained above in the unform way
\begin{equation}\label{theta1nd}
 \sup_{k\in \mathbb{Z}^d}u(k)^2\le \mathrm{K}(\theta)
 \|u\|^{2\theta}
   \|\mathcal{D}^nu\|^{2(1-\theta)},
\end{equation}
where $\mathcal{D}^n$ is as in~\eqref{Dn} and $\theta$
belongs to a certain subinterval of $[0,1]$ uniquely
defined in the corresponding theorem:
\begin{equation}\label{theta1ndinterval}
 \theta\in\left\{
            \begin{array}{ll}
             $[1-1/(2n),1]$, & \hbox{$d=1$, $n\ge 1$;} \\
             $(0,\,1]$, & \hbox{$d=2$, $n=1$;} \\
             $[0,\,1]$, & \hbox{$d\ge 3$, $n=1$.}
            \end{array}
          \right.
\end{equation}

\begin{theorem}\label{T:orth}
Let $\{u^{(j)}\}_{j=1}^N\in l^2(\mathbb{Z}^d)$ be a family of $N$
sequences that are  orthonormal with respect to the natural scalar
product in $l^2(\mathbb{Z}^d)$. We set
\begin{equation}\label{rho}
\rho(k):=\sum_{j=1}^Nu^{(j)}(k)^2,\qquad k\in \mathbb{Z}^d.
\end{equation}
Then for $\theta$ as in \eqref{theta1ndinterval} and
$\theta<1$
\begin{equation}\label{orth}
\|\rho\|_{l^{\frac{2-\theta}{1-\theta}}}^{\frac{2-\theta}{1-\theta}}
=\sum_{k\in\mathbb{Z}^d}
\rho(k)^{\frac{2-\theta}{1-\theta}}\le
\mathrm{K}(\theta)^{\frac1{1-\theta}}
\sum_{j=1}^N\|\mathcal{D}^nu^{(j)}\|^2.
\end{equation}
\end{theorem}
\begin{proof}
For arbitrary $\xi_1,\dots,\xi_N\in\mathbb{R}$ we construct a sequence
$f\in l^2(\mathbb{Z}^d)$
$$
f(k):=\sum_{j=1}^N\xi_ju^{(j)}(k),\qquad k\in\mathbb{Z}^d.
$$
Applying \eqref{theta1nd} and using orthonormality
we obtain for a fixed $k$
$$
f(k)^2\le \mathrm{K}(\theta)\left(\sum_{j=1}^N\xi_j^2\right)^\theta
\left(\sum_{i,j=1}^N\xi_i\xi_j
\bigl(\mathcal{D}^nu^{(i)},\mathcal{D}^nu^{(j)}\bigr)\right)^{1-\theta}.
$$
We now set $\xi_j:=u^{(j)}(k)$:
$$
\rho(k)^2\le \mathrm{K}(\theta)\rho(k)^\theta
\left(\sum_{i,j=1}^Nu^{(i)}(k)u^{(j)}(k)
\bigl(\mathcal{D}^nu^{(i)},\mathcal{D}^nu^{(j)}\bigr)\right)^{1-\theta},
$$
or
$$
\rho(k)^{\frac{2-\theta}{1-\theta}}\le \mathrm{K}(\theta)^\frac1{1-\theta}
\sum_{i,j=1}^Nu^{(i)}(k)u^{(j)}(k)
\bigl(\mathcal{D}^nu^{(i)},\mathcal{D}^nu^{(j)}\bigr).
$$
Summing over $k\in \mathbb{Z}^d$ and using orthonormality
we obtain~\eqref{orth}.
\end{proof}
\begin{corollary}\label{C:N=1}
Setting $N=1$ in Theorem~\ref{T:orth} we obtain a family of
interpolation inequalities for $u\in l^2(\mathbb{Z}^d)$
\begin{equation}\label{N=1}
\|u\|_{l^{\frac{2(2-\theta)}{1-\theta}}(\mathbb{Z}^d)}\le
\mathrm{K}(\theta)^{\frac1{2(2-\theta)}}\|u\|^{\frac1{2-\theta}}
\|\mathcal{D}^nu\|^\frac{1-\theta}{2-\theta}.
\end{equation}
In particular, to mention a few examples with
limiting $\theta$
$$
\aligned
&\|u\|_{l^6(\mathbb{Z})}\le 1\cdot\|u\|^{2/3}\|\mathrm{D}u\|^{1/3},\
\ \theta=1/2,\\
&\|u\|_{l^{10}(\mathbb{Z})}\le 2^{-1/5}\,\|u\|^{4/5}\|\Delta u\|^{1/5},
\ \theta=3/4,
\endaligned
$$
in dimension $d\ge3$
$$
\|u\|_{l^4(\mathbb{Z}^d)}\le \mathrm{K}_d(0)^{1/4}\|u\|^{1/2}\|\nabla u\|^{1/2},\
\ \theta=0.
$$
\end{corollary}
\begin{remark}\label{R:Lad}
{\rm
The last inequality holding in dimesion three and higher curiously resembles the celebrated
Ladyzhenskaya inequality that is vital for the uniqueness
of the weak solutions of the \textit{two-dimensional\/} Navier--Stokes system:
$$
\|f\|_{L_4(\Omega)}\le c_\mathrm{L}\|f\|^{1/2}\|\nabla f\|^{1/2},
\ f\in H^1_0(\Omega), \ \Omega\subseteq\mathbb{R}^2.
$$
}
\end{remark}

We now exploit the equivalence between the inequalities
for orthonormal families and spectral estimates for
the negative trace of the Schr\"{o}dinger operators~\cite{L-T}.

We consider the discrete Schr\"{o}dinger operator
\begin{equation}\label{discrSch}
H:=(-1)^n\Delta^n -V,
\end{equation}
acting on $u\in l^2(\mathbb{Z}^d)$ as follows
$$
Hu(k)=(-1)^n\Delta^nu(k) -V(k)u(k).
$$
\begin{theorem}\label{T:L-T}
Let $V(k)\ge0$ and let $V(k)\to0$ as $|k|\to\infty$, then the negative
spectrum of $H$ is discrete and satisfies the estimate
\begin{equation}\label{discrL-T}
\sum|\lambda_j|\le \mathrm{K}(\theta)\frac{(1-\theta)^{1-\theta}}
{(2-\theta)^{2-\theta}}\sum_{k\in\mathbb{Z}^d}
V(k)^{2-\theta}.
\end{equation}
\end{theorem}
\begin{proof}
Suppose that there exists $N$ negative eigenvalues
$-\lambda_j<0$, $j=1,\dots,N$ with corresponding
$N$ orthonormal eigenfunctions $u^{(j)}$:
$$
(-1)^n\Delta^nu^{(j)}(k) -V(k)u^{(j)}(k)=-\lambda_ju^{(j)}(k).
$$
Taking the scalar product with $u^{(j)}$, summing the
the results with respect to $j$, and using \eqref{rho},
H\"{o}lder inequality and
\eqref{orth}, we obtain
$$
\aligned
\sum_{j=1}^N\lambda_j=(V,\rho)-\sum_{j=1}^N\|\mathcal{D}^nu^{(j)}\|^2\le\\
\|V\|_{l_{2-\theta}}\|\rho\|_{l_{\frac{2-\theta}{1-\theta}}}
-\mathrm{K}(\theta)^{-\frac1{1-\theta}}
\|\rho\|_{l_{\frac{2-\theta}{1-\theta}}}^{\frac{2-\theta}{1-\theta}}\le\\
\max_{y>0}\left(\|V\|_{l_{2-\theta}}y-
\mathrm{K}(\theta)^{-\frac1{1-\theta}}
y^{\frac{2-\theta}{1-\theta}}\right)=\\
\mathrm{K}(\theta)\frac{(1-\theta)^{1-\theta}}
{(2-\theta)^{2-\theta}}\sum_{k\in\mathbb{Z}^d}
V(k)^{2-\theta}.
\endaligned
$$
\end{proof}
\subsection*{Examples}\label{SS:Examp}
\paragraph{$d=1$, $n=1$, $\theta=1/2$.}
Then $\mathrm{K}=1$ and the negative trace of the
operator
$$
-\Delta -V \quad
\text{in}\quad l^2(\mathbb{Z})
$$
satisfies
$$
\sum|\lambda_j|\le \frac2{3\sqrt{3}}\sum_{\alpha=-\infty}^\infty
V^{3/2}(\alpha).
$$
\paragraph{$d=1$, $n=2$, $\theta=3/4$.}
Then $\mathrm{K}=\sqrt{2}/2$ and the negative trace of the
operator
$$
\Delta^2 -V \quad
\text{in}\quad l^2(\mathbb{Z})
$$
satisfies
$$
\sum|\lambda_j|\le \frac{2\sqrt{2}}{5^{5/4}}\sum_{\alpha=-\infty}^\infty
V^{5/4}(\alpha).
$$
\paragraph{$d\ge3$, $n=1$, $\theta=0$.}
Then $\mathrm{K}=\mathrm{K}_d(0)$ and the negative trace of the
operator
$$
-\Delta -V \quad
\text{in}\quad l^2(\mathbb{Z}^d)
$$
satisfies
$$
\sum|\lambda_j|\le \frac{\mathrm{K}_d(0)}{4}\sum_{\alpha\in\mathbb{Z}^d}
V^{2}(\alpha).
$$
In particular, in three dimensions
$$
\frac{\mathrm{K}_3(0)}{4}=0.0631\dots.
$$

\bibliographystyle{amsplain}

\end{document}